\documentclass[a4paper,10pt]{amsart}

\usepackage[
  margin=30mm,
  marginparwidth=25mm,     % + <- Width of your marginpar
  marginparsep=2mm,       % + <- Gap between text block and marginpar
  bottom=25mm,
  ]{geometry}

\usepackage[bbgreekl]{mathbbol}
\usepackage{amsfonts}
\usepackage{latexsym,amssymb,amsthm,mathrsfs,amsmath,amscd,enumerate,enumitem, amscd,color}

\usepackage{tikz}

\DeclareSymbolFontAlphabet{\mathbb}{AMSb}
\DeclareSymbolFontAlphabet{\mathbbl}{bbold}
\newtheorem{ThA}{Theorem}

\newtheorem{thm}{Theorem}[section]
 \newtheorem{cor}[thm]{Corollary}
 \newtheorem{lem}[thm]{Lemma}
 \newtheorem{prop}[thm]{Proposition}
\theoremstyle{definition}

 \theoremstyle{remark}

\usepackage{hyperref}

\newcommand{\N}{\mathbb{N}}

\newcommand{\R}{\mathbb{R}}
\newcommand{\Rn}{\mathbb{R}^n}
\newcommand{\supp}{\mathop{\mathrm{supp}}}

\newcommand{\eps}{\varepsilon}
\numberwithin{equation}{section}
\allowdisplaybreaks
\begin{document}

\footnotetext{Last modification: \today.}

\title[]
 {BMO functions and balayage of Carleson measures in the Bessel setting}

\author[V. Almeida]{V. Almeida}
\address{V\'ictor Almeida, Jorge J. Betancor, Juan C. Fari\~na, Lourdes Rodr\'iguez-Mesa\newline
	Departamento de An\'alisis Matem\'atico, Universidad de La Laguna,\newline
	Campus de Anchieta, Avda. Astrof\'isico S\'anchez, s/n,\newline
	38721 La Laguna (Sta. Cruz de Tenerife), Spain}
\email{valmeida@ull.es,
jbetanco@ull.es, 
jcfarina@ull.es,
lrguez@ull.es}

\author[J. J. Betancor]{J. J. Betancor}

\author[A.J. Castro]{A. J. Castro}
\address{Alejandro J. Castro\newline
	Department of Mathematics, Nazarbayev University, \newline 010000 Astana, Kazakhstan}
\email{alejandro.castilla@nu.edu.kz}

\author[J.C. Fari\~na]{J. C. Fari\~na}

\author[L. Rodr\'{\i}guez-Mesa] {L. Rodr\'{\i}guez-Mesa}

\thanks{The authors are partially supported by MTM2016-79436-P}



\date{}

%%% ----------------------------------------------------------------------

\begin{abstract}
By $BMO_{\textrm{o}}(\mathbb{R})$ we denote the space consisting of all those odd and bounded mean oscillation functions on $\mathbb{R}$. In this paper we characterize the functions in $BMO_{\textrm{o}}(\mathbb{R})$ with bounded support as those ones that can be written as a sum of a bounded function on $(0,\infty )$ plus the balayage of a Carleson measure on $(0,\infty )\times (0,\infty)$ with respect to the Poisson semigroup associated with the Bessel operator
$$
B_\lambda 
:=-x^{-\lambda }\frac{d}{dx}x^{2\lambda }\frac{d}{dx}x^{-\lambda },\quad \lambda >0.
$$
This result can be seen as an extension to Bessel setting of a classical result due to Carleson.
\end{abstract}

%%% ----------------------------------------------------------------------
\maketitle
%%% ----------------------------------------------------------------------

\section{Introduction}

In this paper we extend \cite[Theorem 2]{Ca} (See also \cite[Theorem A]{Wi}) due to Carleson to Bessel settings.

A measurable function $f$ on $\Rn$ is said to have bounded mean oscillation, in short $f \in BMO(\Rn)$, when there exists $C>0$ such that, for every cube $Q \subset \Rn$ with sides parallel to the coordinate axes,
$$\frac{1}{|Q|} \int_Q |f(y)-f_Q| dy
\leq C,$$
where 
$$f_Q
:= \frac{1}{|Q|} \int_Q f(y)dy.$$
It is defined, for every $f \in BMO(\Rn)$,
$$\|f\|_{BMO(\Rn)}
:= \sup_{Q} \frac{1}{|Q|} \int_Q |f(y)-f_Q| dy,$$
where the supremum is taken over all the cubes $Q \subset \Rn$ with sides parallel to the coordinate axis. It is clear that $\| \cdot \|_{BMO(\Rn)}$ is a norm when two functions $f$ and $g$ in $BMO(\Rn)$ are identified provided that the difference $f-g$ is constant in $\Rn$. $BMO(\Rn)$ is also called the John-Nirenberg space (\cite{JN}).

The $BMO(\Rn)$ space is closely connected with the so called Carleson measure in $\R_+^{n+1}:=\Rn \times (0,\infty)$. If $Q$ is a cube in $\Rn$, the Carleson box is given by $\widehat{Q}:=Q \times (0,|Q|)$, where $|Q|$ denotes the Lebesgue measure of $Q$. A Borel measure $\mu$ on $\R_+^{n+1}$ is said to be a Carleson measure, in short $\mu \in \mathcal{C}(\R_+^{n+1})$, when there exists $C>0$ such that, for every cube $Q \subset \Rn$ with sides parallel to the coordinate axes,
$$\frac{|\mu|(\widehat{Q})}{|Q|}
\leq C,$$
where $|\mu|$ denotes the total variation measure of $\mu$.

If $\mu \in \mathcal{C}(\R_+^{n+1})$ it can be defined the norm
$$\|\mu\|_{\mathcal{C}(\R_+^{n+1})}
:= \sup_{Q} \frac{|\mu|(\widehat{Q})}{|Q|},$$
where the supremum is taken  over all cubes $Q \subset \Rn$ with sides parallel to the coordinate axes.

The classical Poisson semigroup $\{P_t\}_{t>0}$ generated by $-\sqrt{- \Delta}$, where $\Delta$ represents the Laplace operator $\Delta := \sum_{i=1}^n \partial_{x_i}^2$ in $\Rn$, is defined for every $f \in L^p(\Rn)$, $1 \leq p \leq \infty$, by
$$P_t(f)(x)
:= c_n \int_{\Rn} P_t(x-y) f(y) dy, \quad x \in \Rn, \ t>0,$$
where $c_n:=\Gamma((n+1)/2)/\pi^{(n+1)/2}$. Here the Poisson kernel is 
$$P_t(x)
:= \frac{t}{(|x|^2+t^2)^{(n+1)/2}}, \quad x \in \Rn, \ t>0.$$

If $f \in BMO(\Rn)$, then (\cite[p. 141]{St})
\begin{equation}\label{eq1}
\int_{\Rn} \frac{|f(y)|}{(1+|y|)^{n+1}} dy < \infty,
\end{equation}
and $P_t(|f|)(x) < \infty$, for every $x \in \Rn$ and $t>0$. 

It is well known (\cite[p. 159 and p. 165]{St}) that a function $f \in L^1(\Rn,(1+|y|)^{-n-1} dy)$ is in $BMO(\Rn)$ if, and only if, the measure $\mu_f$ on $\R_+^{n+1}$ defined by
$$d\mu_f(x,t)
:= | t \partial_t P_t(f)(x) |^2 \frac{dx dt}{t}, \quad (x,t) \in \R_+^{n+1},$$
is a Carleson measure.

If $\mu$ is a positive measure on $\R^{n+1}_+$ the balayage $S_{\mu,P}$ with respect to the Poisson semigroup $\{P_t\}_{t>0}$ is defined by
$$S_{\mu,P}(x)
:= \int_{\R_+^{n+1}} P_t(x-y) d\mu(y,t), \quad x \in \Rn.$$

Carleson (\cite[Theorem 2]{Ca}) (see also \cite[Theorem A]{Wi}) proved that a function $f$ with compact support is in $BMO(\Rn)$  if, and only if, there exist $g \in L^\infty(\Rn)$ and a Carleson measure $\mu$ on $\R_+^{n+1}$ such that $f=g+S_{\mu,P}$ and 
$$\|f\|_{BMO(\Rn)}
\sim \|g\|_{L^\infty(\Rn)} + \|\mu\|_{\mathcal{C}(\R_+^{n+1})}.$$
Actually, this result was established for uniparametric families $\{K_t\}_{t>0}$ being the Poisson semigroup $\{P_t\}_{t>0}$ a special case. An extension of \cite[Theorem 2]{Ca} to spaces of homogeneous type was proved by Uchiyama (\cite{Uc}). The proofs of the mentioned results in \cite{Ca} and \cite{Uc} (see also \cite{GJ}) are based on an iterative argument. Other proof was presented in \cite{Wi}. Here, we will adapt Wilson's ideas to our setting.

Recently, Chen, Duong, Li, Song and Yan (\cite{CDLSY}) have established a version of Carleson's result (\cite[Theorem 2]{Ca}) where the Laplace operator $\Delta$ is replaced by the Schr\"odinger operator $\mathcal{L}_V:=-\Delta + V$, where the nonnegative potential $V$ belongs to the reverse H\"older class $B_q$ for some $q \geq n$. Definitions and main properties about $BMO$ spaces associated with $\mathcal{L}_V$ can be encountered in \cite{DYZ} and \cite{DGMTZ} (see also \cite{Sh}).

Harmonic analysis associated with Bessel operators was initiated by Muckenhoupt and Stein (\cite{MS}). They considered the Bessel operators
$$\Delta_\lambda
:= -x^{-2\lambda} \frac{d}{dx} x^{2\lambda} \frac{d}{dx}, \quad \lambda >0,$$
and studied $L^p$-boundedness properties of maximal operators associated with Poisson semigroups defined by $\Delta_\lambda$ and Riesz transforms in this setting. Recently, harmonic analysis related to Bessel operators has raised interest again (see \cite{BCS},
\cite{DLOWY}, 
\cite{DLWY}, 
\cite{LW},
\cite{NS1}, 
\cite{NS2}, 
\cite{Vi}
and \cite{YY},
 among others).

We consider the Bessel type operator on $(0,\infty)$
$$B_\lambda
:= - x^{-\lambda} \frac{d}{dx} x^{2\lambda} \frac{d}{dx} x^{-\lambda}
= - \frac{d^2}{dx^2} + \frac{\lambda (\lambda-1)}{x^2}, \quad \lambda >0.$$
Note that the potential $V_\lambda(x):=\lambda(\lambda-1)/x^2$, $x \in (0,\infty)$, does not satisfy any reverse H\"older property and it has a singularity at $x=0$. Then, $B_\lambda$ is not included in the class of Schr\"odinger operators considered in \cite{CDLSY} and  \cite{DGMTZ}.

Assume that $\lambda>0$. According to \cite[\S 16]{MS} the Poisson semigroup $\{P_t^\lambda\}_{t>0}$ associated with the Bessel operator $B_\lambda$ is defined as follows
$$P_t^\lambda(f)(x)
:= \int_0^\infty P_t^\lambda(x,y)f(y) dy, \quad x,t \in (0,\infty),$$
for every $f \in L^p(0,\infty)$, $1 \leq p \leq \infty$. Here, the Poisson kernel $P_t^\lambda(x,y)$ is given by
$$P_t^\lambda(x,y)
:= \frac{2\lambda}{\pi} (xy)^\lambda t \int_0^\pi \frac{(\sin \theta)^{2\lambda-1}}{[(x-y)^2+t^2+2xy(1-\cos \theta)]^{\lambda+1}} d\theta,
\quad x,y,t \in (0,\infty).$$
For every $1 \leq p \leq \infty$, the family $\{P_t^\lambda\}_{t>0}$ is bounded on $L^p(0,\infty)$. Note that $\{P_t^\lambda\}_{t>0}$ is not Markovian, that is, $P_t^\lambda(1) \neq 1$. Indeed, according to \cite[Lemma 2.2 and Remark 2.5]{BSt} the function $v(x,t):=P_t^\lambda(1)(x)$ satisfies
$$(\partial_t^2 - B_{\lambda,x})v(x,t)=0, \quad x,t \in (0,\infty),$$
but clearly 
$$(\partial_t^2 - B_{\lambda,x})1 = -\frac{\lambda (\lambda-1)}{x^2}, \quad x,t \in (0,\infty).$$ 
We also remark that the function ${\mathfrak f}_\lambda(x)=x^\lambda$, $x \in (0,\infty)$, does not belong to $L^p(0,\infty)$, for any $1 \leq p \leq \infty$. However, $P_t^\lambda({\mathfrak f}_\lambda)={\mathfrak f}_\lambda$, $t>0$ (see, \cite[p. 455]{BCaFR}).  

We denote by $P_*^\lambda$ the maximal operator defined by $\{P_t^\lambda\}_{t>0}$, that is, 
$$P_*^\lambda(f)
:= \sup_{t>0} |P_t^\lambda(f)|, \quad f \in L^p(0,\infty), \quad 1 \leq p \leq \infty.$$
$P_*^\lambda$ is bounded from $L^p(0,\infty)$ into itself when $1<p \leq \infty$ and from $L^1(0,\infty)$ into $L^{1,\infty}(0,\infty)$ (\cite[Theorem 2.4 and Remark 2.5]{BSt}).

The Hardy space $H^1_\lambda(0,\infty)$ associated to the operator $B_\lambda$ was studied in \cite{BDT}. It is said that a function $f \in L^1(0,\infty)$ is in $H^1_\lambda(0,\infty)$ when $P_*^\lambda(f) \in L^1(0,\infty)$. On $H^1_\lambda(0,\infty)$ it is considered the norm $\|\cdot \|_{H^1_\lambda(0,\infty)}$ given by
$$\|f \|_{H^1_\lambda(0,\infty)}
:= \|f \|_{L^1(0,\infty)} + \|P_*^\lambda(f) \|_{L^1(0,\infty)}, \quad f \in H^1_\lambda(0,\infty).$$
The dual space of $H^1_\lambda(0,\infty)$ can be characterized as a $BMO$-type space. A function $f \in L^1(0,a)$, for every $a>0$, is in $BMO_{\textrm{o}}(\mathbb{R})$ when there exists $C>0$ such that
\begin{itemize}
\item[$a)$] for every bounded interval $I \subset (0,\infty)$,
$$\frac{1}{|I|} \int_I |f(y)-f_I| dy \leq C,$$
\item[$b)$] for every $a \in (0,\infty)$, 
$$\frac{1}{a} \int_0^a |f(y)| dy \leq C.$$
\end{itemize}
On $BMO_{\textrm{o}}(\mathbb{R})$ the norm $\|\cdot \|_{BMO_{\textrm{o}}(\mathbb{R})}$ is defined by
$$
\| f \|_{BMO_{\textrm{o}}(\mathbb{R})}
:= \inf \{C>0 \text{ :  {\it a}) and {\it b}) hold}\}.
$$
The space $BMO_{\textrm{o}}(\mathbb{R})$ can be characterized as that one consisting on all the functions $f$ defined on $(0,\infty)$ such that the odd extension $f_{\textrm{o}}$ of $f$ to $\R$ is in $BMO(\R)$ (\cite[p.465]{BCFR1}). This property, that justifies the notation $BMO_{\textrm{o}}(\mathbb{R})$ for our space, will be very useful in the sequel. The space $BMO_{\textrm{o}}(\mathbb{R})$ coincides, in the usual way, with the dual space of $H^1_\lambda(0,\infty)$ (see \cite[p. 466]{BCFR1}).

We say that a Borel measure $\mu$ on $(0,\infty)\times (0,\infty )$ is a Carleson measure on $(0,\infty)\times (0,\infty )$ when there exists $C>0$ such that, for every bounded interval $I \subset (0,\infty)$,
$$\frac{|\mu|(\widehat{I})}{|I|} \leq C.$$
Here, as above, $|\mu|$ represents the total variation measure of $\mu$, $|I|$ denotes the length of the interval $I$ and $\widehat{I}:=I \times (0,|I|)$. If $\mu$ is a Carleson measure on $(0,\infty)\times (0,\infty )$ we define
$$\|\mu\|_{\mathcal{C}}
:= \sup_{I} \frac{|\mu|(\widehat{I})}{|I|},$$
where the supremum is taken over all bounded intervals $I \subset (0,\infty)$.

Next result shows the connection between $BMO_{\textrm{o}}(\mathbb{R})$ and the Carleson measures on $(0,\infty)\times (0,\infty )$ by using Poisson semigroups $\{P_t^\lambda\}_{t>0}$.

\begin{ThA}\label{Th1.1} {\normalfont (\cite[Theorem 1.1]{BCFR1})}
Let $\lambda >0$. Suppose that $f \in L^1(0,a)$, for every $a>0$. Then, $f \in BMO_{\textrm{o}}(\mathbb{R})$ if, and only if, $f \in L^1((0,\infty),(1+x)^{-2} dx)$ and the measure $\mu_f$ on $(0,\infty)\times (0,\infty )$ defined by
$$\mu_f(x,t)
:= | t \partial_t P_t^\lambda(f)(x) |^2 \frac{dxdt}{t}, \quad x,t \in (0,\infty), $$
is Carleson. Moreover, the quantities $\|f\|_{BMO_{\textrm{o}}(\mathbb{R})}^2$ and $\|\mu_f\|_{\mathcal{C}}$ are equivalent.
\end{ThA}

In our $B_\lambda$-Bessel setting we consider the gradient 
$\nabla_{\lambda}
:=(\partial_t,D_{\lambda ,x})$, where $D_{\lambda ,x}:=x^\lambda \partial_x x^{-\lambda}$  .

\begin{ThA}\label{Th1.2} {\normalfont (\cite[Theorem 1]{BCaFR})}
Let $\lambda >1$. Assume that $u$ is a function defined in $\R \times (0,\infty)$ such that
$x^{-\lambda} u(x,t) \in C^\infty(\R \times (0,\infty))$ and it is even in the $x$-variable. Suppose also that $(\partial_t^2 - B_\lambda)u=0$, on $(0,\infty)\times (0,\infty )$. Then, the following assertions are equivalent.
\begin{itemize}
\item[$(i)$] There exists $f \in BMO_{\textrm{o}}(\mathbb{R})$ such that 
$u(x,t)=P_t^{\lambda}(f)(x)$, $x,t \in (0,\infty)$.
\item[$(ii)$] The measure $\mu_\lambda$ on $(0,\infty)\times (0,\infty )$ defined by
$$d\mu_\lambda(x,t)
:= | t \nabla_\lambda u(x,t) |^2 \frac{dxdt}{t}, \quad x,t \in (0,\infty), $$
is Carleson.
\end{itemize}
Moreover, the quantities $\|f\|_{BMO_{\textrm{o}}(\mathbb{R})}^2$ and $\|\mu_\lambda\|_{\mathcal{C}}$ are equivalent.
\end{ThA}

The main result of this paper is the following, which can be seen as a version of the Carleson's result in \cite[Theorem 2]{Ca} (see also \cite[Theorem A]{Wi}) in our Bessel setting.

\begin{thm}\label{Th1.3}
Let $\lambda >0$. 
\begin{itemize}
\item[$(i)$] If $\mu$ is a Carleson measure on $(0,\infty)\times (0,\infty )$, the balayage of $\mu$ with respect to the Poisson semigroup $\{P_t^\lambda\}_{t>0}$ associated with $B_\lambda$ defined by
$$S_{\mu,P^\lambda}(x)
:= \int_0^\infty\int_0^\infty P_t^\lambda(x,y) d\mu(y,t), \quad x \in (0,\infty),$$
is in $BMO_{\textrm{o}}(\mathbb{R})$ and 
$$\|S_{\mu,P^\lambda}\|_{BMO_{\textrm{o}}(\mathbb{R})}
\leq C \|\mu\|_{\mathcal{C}}.$$
Here $C>0$ does not depend on $\mu$.

\item[$(ii)$] Let $f \in BMO_{\textrm{o}}(\mathbb{R})$ such that $f=0$ on $(a,\infty)$, for some $a>0$. Then, there exist $g \in L^\infty(0,\infty)$ and a Carleson measure $\mu$ on $(0,\infty)\times (0,\infty )$ such that $f = g + S_{\mu,P^\lambda}$ and 
$$\|g\|_{L^\infty(0,\infty)} + \|\mu\|_{\mathcal{C}}
\leq C \|f\|_{BMO_{\textrm{o}}(\mathbb{R})},$$
where $C>0$ does not depend on $f$.
\end{itemize}
\end{thm}

In order to prove this theorem we are going to adapt the procedure developed by Wilson (\cite{Wi}) to our Bessel setting.

The heat semigroup $\{W_t^\lambda\}_{t>0}$ associated to the Bessel operator $B_\lambda$ is defined, for every $f \in L^p(0,\infty)$, $1 \leq p \leq \infty$, by
$$W_t^\lambda(f)(x)
:= \int_0^\infty W_t^\lambda(x,y)f(y)dy, \quad x \in (0,\infty),$$
where the heat kernel is given by 
$$W_t^\lambda(x,y)
:= \frac{\sqrt{xy}}{2t} I_{\lambda-1/2}\Big( \frac{xy}{2t}\Big) e^{-(x^2+y^2)/(4t)}, \quad x,y,t \in (0,\infty).$$
Here, $I_\nu$ denotes the modified Bessel function of the first kind and order $\nu$. If $\mu$ is a Borel measure on $(0,\infty)\times (0,\infty )$ we define the balayage $S_{\mu,W^\lambda}$ of $\mu$ with respect to $\{W_t^\lambda\}_{t>0}$ in the natural way.

The well known subordination formula connects Bessel Poisson and heat semigroups as follows. For every $f \in L^p(0,\infty)$, $1 \leq p \leq \infty$, 
$$P_t^\lambda(f)(x)
= \frac{1}{\sqrt{\pi}} \int_0^\infty \frac{e^{-u}}{\sqrt{u}} W^\lambda_{t^2/(4u)}(f)(x) du, \quad x,t \in (0,\infty). $$
By using this equality from Theorem \ref{Th1.3} we can immediately deduce the following property (see \cite[proof of Theorem 3.5]{CDLSY}).

\begin{cor}\label{Cor1.1}
Let $\lambda >0$ and $f \in BMO_{\textrm{o}}(\mathbb{R})$ such that $f=0$ on $(a,\infty)$, for some $a>0$. Then, there exist $g \in L^\infty(0,\infty)$ and a Carleson measure $\mu$ on $(0,\infty)\times (0,\infty )$ such that $f = g + S_{\mu,W^\lambda}$ and 
$$\|g\|_{ L^\infty(0,\infty)} + \|\mu\|_{\mathcal{C}}
\leq C \|f\|_{BMO_{\textrm{o}}(\mathbb{R})}.$$
Here $C>0$ does not depend on $f$.
\end{cor}

This paper is organized as follows. In Section \ref{Sect2}, we present some properties of the Poisson kernel and Poisson semigroups for Bessel operators that will be useful in the sequel. In Section \ref{Sect3} we prove new properties of the space $BMO_{\textrm{o}}(\mathbb{R})$ that are needed to establish Theorem \ref{Th1.3}. The proof of Theorem \ref{Th1.3} is presented in Section \ref{Sect4}.

Throughout this paper by $C$ we always denote a positive constant that is not necessarily the same in each occurrence. Also, we always consider $\lambda>0$.

%%% ----------------------------------------------------
\section{Some useful properties of Bessel Poisson semigroups}\label{Sect2}

As it was mentioned in the introduction, according to  \cite[\S 16]{MS} the $B_\lambda$-Poisson kernel is given by
\begin{equation}\label{D1}
P_t^\lambda(x,y)
:= \frac{2\lambda}{\pi} t (xy)^\lambda \int_0^\pi \frac{(\sin \theta)^{2\lambda-1}}{[(x-y)^2+t^2+2xy(1-\cos \theta)]^{\lambda+1}} d\theta,
\quad x,y,t \in (0,\infty).
\end{equation}
From \eqref{D1} it is straightforward that 
\begin{equation}\label{D2}
0 
\leq P_t^\lambda(x,y)
\leq C \frac{t (xy)^\lambda}{((x-y)^2+t^2)^{\lambda + 1}},
\quad x,y,t \in (0,\infty).
\end{equation}
Also, by \cite[(b) p. 86]{MS} we get
\begin{equation}\label{D3}
P_t^\lambda(x,y)
\leq C \frac{t}{(x-y)^2+t^2},
\quad x,y,t \in (0,\infty).
\end{equation}

We also need estimations for the derivatives of $P_t^\lambda(x,y)$. 

\begin{lem}\label{D4}
Let $\lambda>0$. Then, for every $t,x,y \in (0,\infty)$,
$$
|\partial_t P_t^\lambda(x,y)|+|D_{\lambda ,x}P_t^\lambda(x,y)|
\leq \frac{C}{t}  P_t^\lambda(x,y), \qquad 
$$
\end{lem}

\begin{proof}
We have that
\begin{align*}
| \partial_t P_t^\lambda(x,y) |
& = \Big| \frac{2\lambda}{\pi} (xy)^\lambda \Big\{ 
\int_0^\pi \frac{(\sin \theta)^{2\lambda-1}}{[(x-y)^2+t^2+2xy(1-\cos \theta)]^{\lambda+1}} d\theta \nonumber \\
& \qquad \qquad \quad - 2(\lambda +1)t^2 \int_0^\pi \frac{(\sin \theta)^{2\lambda-1}}{[(x-y)^2+t^2+2xy(1-\cos \theta)]^{\lambda+2}} d\theta\Big\} \Big| \nonumber \\
& \leq \frac{C}{t}  P_t^\lambda(x,y).
\end{align*}

On the other hand, we can write
\begin{align*} 
D_{\lambda,x} P_t^\lambda(x,y)
& = -\frac{4\lambda(\lambda+1)}{\pi} t(xy)^\lambda   
\int_0^\pi \frac{((x-y)+y(1-\cos \theta))(\sin \theta )^{2\lambda -1}}{[(x-y)^2+t^2+2xy(1-\cos \theta)]^{\lambda+2}} d\theta,
\end{align*}
and also, 
\begin{align*} 
D_{\lambda,x}P_t^\lambda(x,y) 
& = -\frac{4\lambda(\lambda+1)}{\pi} t(xy)^\lambda   
\int_0^\pi \frac{((x-y)\cos \theta +x(1-\cos \theta))(\sin \theta)^{2\lambda-1}}{((x-y)^2+t^2+2xy(1-\cos \theta))^{\lambda+2}} d\theta.
\end{align*}

Then,
$$
\Big| D_{\lambda,x} P_t^\lambda(x,y) \Big|\leq Ct(xy)^\lambda \int_0^\pi \frac{(|x-y|+\min\{x,y\}(1-\cos \theta))(\sin \theta)^{2\lambda-1}}{[(x-y)^2+t^2+2xy(1-\cos \theta)]^{\lambda+2}} d\theta .
$$

Since
$$
\frac{|x-y|+\min\{x,y\}(1-\cos \theta)}{(x-y)^2+t^2+2xy(1-\cos \theta)}\leq \frac{C}{t}\Big(1+\frac{\min \{x,y\}}{\sqrt{xy}}\Big)\leq \frac{C}{t},\quad x,y,t\in (0,\infty )\mbox{ and }\theta \in [0,\pi], 
$$
we conclude that
\begin{align*} 
\Big| D_{\lambda,x} P_t^\lambda(x,y)  \Big|
& \leq \frac{C}{t}  P_t^\lambda(x,y).
\end{align*}
\end{proof}

The Hankel transform $h_\lambda(f)$ of $f \in L^1(0,\infty)$ is defined by
$$
h_\lambda(f)(x):= \int_0^\infty \sqrt{xy} J_{\lambda-1/2}(xy) f(y) dy, \quad x \in (0,\infty),
$$
where $J_\nu$ denotes the Bessel function of the first kind and order $\nu$. Since the function $\sqrt{z} J_{\lambda-1/2}(z)$ is bounded on $(0,\infty)$, it is clear that 
$$\| h_\lambda(f) \|_{L^\infty(0,\infty)}
\leq C \|f\|_{L^1(0,\infty)}, \quad f \in L^1(0,\infty).$$
The Hankel transform $h_\lambda$ can be extended from $L^1(0,\infty) \cap L^2(0,\infty)$ to $L^2(0,\infty)$ as an isometry in $L^2(0,\infty)$ (\cite[p. 473 (1)]{Tit}).

The Bessel Poisson kernel can be written in the following way (\cite[(16.1')]{MS})
\begin{equation*} 
P_t^\lambda(x,y)
= \int_0^\infty e^{-tz} \sqrt{xz} J_{\lambda-1/2}(xz) \sqrt{yz} J_{\lambda-1/2}(yz) \, dz, \quad x,y,t\in (0,\infty).
\end{equation*}
Then, we have that
\begin{equation*}%\label{D6}
P_t^\lambda(x,y)
= h_\lambda \Big( e^{-tz} \sqrt{xz} J_{\lambda-1/2}(xz) \Big)(y), 
\quad x,y,t\in (0,\infty),
\end{equation*}
and
\begin{equation}\label{A.4}
\partial_t P_t^\lambda (x,y)
= - h_\lambda\Big(z e^{-tz}\sqrt{xz} J_{\lambda-1/2}(xz) \Big)(y), \quad  x,y,t\in (0,\infty).
\end{equation}

We can also obtain that if $f\in L^2(0,\infty)$, 
\begin{equation}\label{D7}
P_t^\lambda (f)
= h_\lambda\Big(e^{-tz}h_\lambda(f)(z)\Big), \quad t \in (0,\infty),
\end{equation}
and
\begin{equation}\label{A.3}
\partial_t P_t^\lambda (f)
= - h_\lambda\Big(z e^{-tz}h_\lambda(f)(z)\Big), \quad t \in (0,\infty).
\end{equation}

Indeed, since the function $\sqrt{z} J_\lambda(z)$ is bounded on $(0,\infty)$, we get
\begin{align*}
& \int_0^\infty |\sqrt{yz}  J_{\lambda-1/2}(yz) e^{-tz}h_\lambda(f)(z) |
\, dz 
 \leq C \int_0^\infty  e^{-tz} |h_\lambda(f)(z) |\, dz \\
& \qquad \qquad  \leq C \Big(\int_0^\infty e^{-2tz} \, dz \Big)^{1/2} \|h_\lambda(f)\|_{L^2(0,\infty)} \leq \frac{C}{t^{1/2}} \|f\|_{L^2(0,\infty)} < \infty,  \quad y, t \in (0,\infty),
\end{align*}
which allows us to establish \eqref{D7}. In analogous way the differentiation under the integral sign in \eqref{A.3} can be justified. 

Also, since (see \cite[(5.3.5)]{Le}), 
$$
\partial_y [(yz)^{-\nu} J_\nu(yz)]
= -z(yz)^{-\nu} J_{\nu+1}(yz), \quad y,z  \in (0,\infty),
$$
it follows that
\begin{align}\label{DyKernel}
D_{\lambda ,y}[P_t^\lambda (x,y)]
& = y^\lambda \partial_y
\int_0^\infty (yz)^{-\lambda + 1/2} J_{\lambda-1/2}(yz) z^\lambda e^{-tz} \sqrt{xz} J_{\lambda-1/2}(xz)  \, dz \nonumber\\
& =- 
\int_0^\infty \sqrt{yz} J_{\lambda+1/2}(yz) z e^{-tz} \sqrt{xz} J_{\lambda-1/2}(xz)  \, dz \nonumber\\
& = -h_{\lambda +1}\Big( z e^{-tz} \sqrt{xz} J_{\lambda-1/2}(xz) \Big) (y), \quad x,y, t \in (0,\infty),
\end{align}
and, for $f\in L^2(0,\infty )$,
\begin{equation}\label{DyPoisson}
D_{\lambda ,y}[ P_t^\lambda(f)(y)]
= -h_{\lambda +1}\Big( z e^{-tz} h_\lambda(f)(z) \Big) (y), \quad y,t \in (0,\infty).
\end{equation}
Differentiation under the integral sign can be justified as above. 

On the other hand, since $h_\lambda(f) \in L^2(0,\infty)$, the dominated convergence theorem implies that 
$$
\lim_{t \to 0^+}tz e^{-tz} h_\lambda(f)(z) = \lim_{t \to +\infty}tz e^{-tz} h_\lambda(f)(z) =\lim_{t \to +\infty}e^{-tz}h_\lambda(f)(z)=0, \quad 
\text{ in } L^2(0,\infty),
$$
and
$$
\lim_{t \to 0^+}e^{-tz}h_\lambda(f)(z)=h_\lambda(f)(z),\quad \text{ in } L^2(0,\infty).
$$

Then, from \eqref{D7}, \eqref{A.3} and the $L^2$-boundedness of $h_\lambda$ we get that
\begin{equation}\label{D8}
\lim_{t\to 0^+}t \partial_t (P_{t}^\lambda(f)(z))
= \lim_{t\to +\infty}t \partial_t (P_{t}^\lambda(f)(z))=\lim_{t\to +\infty} P_{t}^\lambda(f)(z)=0, \quad \text{ in } L^2(0,\infty),
\end{equation}
and
\begin{equation}\label{A.4.1}
\lim_{t\to 0^+}P_t^\lambda(f)(z)=f, \quad \text{ in } L^2(0,\infty).
\end{equation}
Our next objective is to prove the following lemma. 
\begin{lem}
Let $f\in L^2(0,\infty)$. Then,
\begin{equation}\label{A.2}
f(x)
= 2 \lim_{\eps \to 0^+} \int_{\eps}^{1/\eps} \int_0^\infty t \nabla_{\lambda,y} (P_t^\lambda (x,y)) \cdot \nabla_{\lambda,y}(P_t^\lambda (f)(y)) \, dy \, dt, 
\end{equation}
where the equality is understood in $L^2(0,\infty)$ and also in a distributional sense.
\end{lem}

\begin{proof}

By using \eqref{A.4} and \eqref{A.3}, Plancherel inequality for $h_\lambda$ leads to
\begin{align*}
\int_0^\infty \partial_t P_t^\lambda (x,y) \, \partial_t P_t^\lambda(f)(y) dy
& = \int_0^\infty h_\lambda\Big(z e^{-tz}\sqrt{xz} J_{\lambda-1/2}(xz) \Big)(y) \, h_\lambda\Big(z e^{-tz}h_\lambda(f)(z)\Big)(y) dy \\
& = \int_0^\infty z^2 e^{-2tz}\sqrt{xz} J_{\lambda-1/2}(xz)h_\lambda(f)(z) dz \\
& = \frac{1}{4} \partial_t^2 [P_{2t}^\lambda(f)(x)], \quad x, t \in (0,\infty).
\end{align*}

In analogous way from \eqref{DyKernel}, \eqref{DyPoisson} and Plancherel equality for $h_{\lambda + 1}$ we obtain
\begin{align*}
& \int_0^\infty 
D_{\lambda ,y}[P_t^\lambda (x,y)] \, 
D_{\lambda ,y}[P_t^\lambda(f)(y)] \, dy \\
& \qquad \qquad = 
\int_0^\infty h_{\lambda+1}\Big(z e^{-tz}\sqrt{xz} J_{\lambda-1/2}(xz) \Big)(y) \, h_{\lambda+1}\Big(z e^{-tz}h_\lambda(f)(z)\Big)(y) dy \\
& \qquad \qquad  = \int_0^\infty z^2 e^{-2tz} h_\lambda(f)(z) \sqrt{xz} J_{\lambda-1/2}(xz) dz \\
& \qquad \qquad  = \frac{1}{4} \partial_t^2 [P_{2t}^\lambda(f)(x)], \quad x, t \in (0,\infty).
\end{align*}

By partial integration we obtain
\begin{align*}
&\int_{\eps}^{1/\eps} \int_0^\infty t \nabla_{\lambda,y} (P_t^\lambda (x,y)) \cdot \nabla_{\lambda,y}(P_t^\lambda (f)(y))\, dy \, dt
 = \frac{1}{2} \int_{\eps}^{1/\eps} t \partial_t^2 [P_{2t}^\lambda(f)(x)] \, dt \\
& \qquad \qquad  = \frac{1}{2}  \Big\{ 
t \partial_t (P_{2t}^\lambda(f)(x)) \Big]_{t=\eps}^{t=1/\eps} 
- P_{2t}^\lambda(f)(x)\Big]_{t=\eps}^{t=1/\eps} \Big\},
\quad x \in (0,\infty) \text{ and } 0<\eps<1.
\end{align*}
We conclude from \eqref{D8} and \eqref{A.4.1} that
$$
\lim_{\eps \to 0^+} \int_{\eps}^{1/\eps} \int_0^\infty t \nabla_{\lambda,y} (P_t^\lambda (x,y)) \cdot \nabla_{\lambda,y}(P_t^\lambda (f)(y)) \, dy \, dt
= \frac{f(x)}{2}, \quad
\text{in } L^2(0,\infty),
$$
and then, also, in a distributional sense.
\end{proof}

%%% ----------------------------------------------------
\section{$BMO$ spaces associated with Bessel operators}\label{Sect3}
In this section we establish some properties for the functions in the space $BMO_{\textrm{o}}(\mathbb{R})$ that will be useful in the proof of Theorem \ref{Th1.3}.

As it was mentioned in the introduction, in \cite{BDT} Hardy spaces associated with Bessel operators $B_\lambda $ were introduced by using maximal operators. A function $f\in L^1(0,\infty )$ is in the Hardy space $H_\lambda ^1(0,\infty )$ provided that the maximal function $P_*^\lambda (f) \in L^1 (0,\infty )$. Here, $P_*^\lambda $ is defined by
$$
P_*^\lambda (f)
:=\sup_{t>0}|P_t^\lambda (f)|,\quad f\in L^1(0,\infty ).
$$
According to \cite[Theorem 1.10 and Proposition 3.8]{BDT} $H_\lambda^1(0,\infty )$ can be also defined by using the maximal operator associated to the heat semigroup $\{W_t^\lambda \}_{t>0}$ generated by $-B_\lambda$.

The area integral defined by the Poisson semigroup $\{P_t^\lambda\}_{t>0}$, $g_\lambda (f)$ of $f\in L^1 (0,\infty )$ is defined by
$$
g_\lambda (f)(x)
:=\Big(\int_{\Gamma _+(x)}|t\partial _tP_t^\lambda (f)(y)|^2\frac{dtdy}{t^2}\Big)^{1/2},\quad x\in (0,\infty ),
$$
where $\Gamma _+(x):=\{(y,t)\in (0,\infty )\times (0,\infty ): |x-y|<t\}$, $x\in (0,\infty )$. In \cite[Proposition 4.1]{BCFR1} it was proved that $f\in L^1(0,\infty )$ is in $H_\lambda^1(0,\infty )$ if and only if $g_\lambda (f)\in L^1(0,\infty )$. Actually, the space $H_\lambda ^1(0,\infty )$ does not depend on $\lambda $ because, according to \cite[Theorem 1.10]{BDT} and \cite[Theorem 2.1]{Fri}, a function $f\in L^1 (0,\infty )$ is in $H_\lambda ^1(0,\infty )$ when and only when the odd extension $f_{\textrm{o}}$ of $f$ to $\mathbb{R}$ is in the classical Hardy space $H^1(\mathbb{R})$. Other characterizations for the space $H_\lambda ^1(0,\infty )$ can be found in \cite{BDLWY} (even in the multiparametric case).

The dual space of $H_\lambda ^1(0,\infty )$ is the space $BMO_{\textrm{o}}$ (\cite[p. 466]{BCFR1}). By using duality and the description of $H_\lambda ^1(0,\infty )$ in terms of $g_\lambda$ we deduce a new characterization of $BMO_{\textrm{o}}(\mathbb{R})$.

According to (\ref{D3}) we have that $P_t^\lambda (f)(x)<\infty$, for every $x,t\in (0,\infty )$, provided that $f$ is a complex measurable function on $(0,\infty )$ such that
$$
\int_0^\infty \frac{|f(x)|}{(1+x)^2}dx<\infty .
$$
We say that a function $f\in L^1((0,\infty ),(1+x)^{-2}dx)$ is in $BMO(P^\lambda )$ when
$$
\|f\|_{BMO(P^\lambda )}:=\sup \frac{1}{|I|}\int_I|f(x)-P_{|I|}^\lambda (f)(x)|dx<\infty,
$$
where the supremum is taken over all bounded intervals $I$ in $(0,\infty )$.

We now characterize $BMO(P^\lambda )$ as the dual space of $H^1_\lambda (0,\infty )$. In order to do this we consider the odd-atoms introduced in \cite{Fri}. A measurable function $\mathfrak a$ on $(0,\infty )$ is an odd-atom when it satisfies one of the following properties:

$(a)$ $\mathfrak a=\frac{1}{\delta}\chi _{(0,\delta)}$, for some $\delta >0$. Here $\chi _{(0,\delta)}$ denotes the characteristic function of $(0,\delta)$, for every $\delta >0$.

$(b)$ There exists a bounded interval $I\subset (0,\infty)$ such that $\supp \mathfrak a\subset I$, $\int_I\mathfrak a(x)dx=0$ and $\|\mathfrak a\|_\infty \leq |I|^{-1}$.

We say that a function $f\in L^1(0,\infty )$ is in $H_{\textrm{o}, at}^1(0,\infty )$ when, for every $j\in\mathbb{N}$, there exist $\lambda _j>0$ and an odd-atom $\mathfrak a_j$ such that $f=\sum_{j\in \mathbb{N}}\lambda _j\mathfrak a_j$, in $L^1(0,\infty)$, and $\sum_{j\in \mathbb{N}}\lambda _j<\infty$. We define, for every $f\in H_{\textrm{o}, at}^1(0,\infty )$,
$$
\|f\|_{H_{\textrm{o}, at}^1(0,\infty )}
:=\inf\sum_{j\in \mathbb{N}}\lambda _j,
$$
where the infimum is taken over all the sequences $\{\lambda_j\}_{j\in \mathbb{N}}\subset (0,\infty )$ such that $\sum_{j\in\mathbb{N}}\lambda _j<\infty$ and $f=\sum_{j\in \mathbb{N}}\lambda_j\mathfrak a_j$, in $L^1(0,\infty )$, where $\mathfrak a_j$ is an odd-atom, for every $j\in \mathbb{N}$. 

According to \cite[Proposition 3.7]{BDT} we have that $H_{\textrm{o}, at}^1(0,\infty )=H_\lambda ^1(0,\infty )$ algebraic and topologically. Note that this equality implies that $\mathcal{A}=\mbox{span}\{\mbox{odd atoms}\}$ is a dense subspace of $H_\lambda ^1(0,\infty )$.

We now characterize $BMO(P^\lambda )$ as the dual space of $H_\lambda ^1(0,\infty )$.

\begin{prop}\label{Prop2.1}
Let $\lambda >0$.
\begin{itemize}
\item[$i)$] Let $f\in BMO(P^\lambda )$. We define the functional $T_f$ on $\mathcal{A}$ by
$$
T_f(b)
:=\int_0^\infty f(x)b(x)dx,\quad b\in \mathcal{A}.
$$
Then, $T_f$ can be extended to $H_\lambda ^1(0,\infty )$ as a bounded operator from $H_\lambda ^1(0,\infty )$ into $\mathbb{C}$. Furthermore,
$$
\|T_f\|_{(H_\lambda ^1(0,\infty ))'}\leq C\|f\|_{BMO(P^\lambda)},
$$
where $C>0$ does not depend on $f$.

\item[$ii)$] There exists $C>0$ such that, for every $T\in (H_\lambda ^1(0,\infty ))'$, there exists $f\in BMO(P^\lambda )$ such that $T=T_f$ on $\mathcal{A}$ and
$$
\|f\|_{BMO(P^\lambda )}\leq C\|T\|_{(H_\lambda ^1(0,\infty ))'}.
$$
\end{itemize}
\end{prop}

\begin{proof}[Proof of Proposition \ref{Prop2.1}, $i)$]
Since $f\in BMO(P^\lambda )$, we can affirm that the measure $\rho _f$ on $(0,\infty )\times (0,\infty )$ defined by
$$
d\rho _f(x,t)
:=|t(\partial _tP_t^\lambda )(i_d -P_t^\lambda )f(x)|^2\frac{dxdt}{t},
$$
where $i_d$ represents the identity operator, is Carleson and 
$$
\|\rho _f\|_{\mathcal{C}}\leq C\|f\|^2_{BMO(P^\lambda )},
$$
with $C>0$.

In order to prove this assertion, we can proceed as in the proof of \cite[Lemma 4.6]{DY}. Indeed, it is sufficient to see that, there exists $C>0$ such that, for every bounded interval $I\subset (0,\infty )$ we have that
\begin{equation}\label{2.1}
\int_{\widehat{I}}|t(\partial_tP_t^\lambda) (i_d-P_t^\lambda )(i_d-P_{|I|}^\lambda)f(x)|^2\frac{dxdt}{t}\leq C|I|\|f\|_{BMO(P^\lambda )}^2,
\end{equation}
and
\begin{equation}\label{2.2}
\int_{\widehat{I}}|t(\partial_tP_t^\lambda )(i_d-P_t^\lambda )P_{|I|}^\lambda (f)(x)|^2\frac{dxdt}{t}\leq C|I|\|f\|_{BMO(P^\lambda )}^2.
\end{equation}
We consider the Littlewood-Paley type function $G_\lambda$ defined by
$$
G_\lambda (g)(x)
:=\Big(\int_0^\infty |t(\partial _tP_t^\lambda )(i_d-P_t^\lambda )g(x)|^2\frac{dt}{t}\Big)^{1/2},\quad x\in (0,\infty ).
$$

Let $g\in L^2(0,\infty )$. According to \eqref{D7}, \eqref{A.3} and the fact that $h_\lambda ^2=i_d$, we can write
$$
t(\partial _tP_t^\lambda )(i_d-P_t^\lambda )g=-h_\lambda (tze^{-tz}(1-e^{-tz})h_\lambda (g)),\quad t>0.
$$
Then, since $h_\lambda $ is a bounded operator from $L^2(0,\infty )$ into itself, we get
\begin{align*}
\|G_\lambda (g)\|_2
&=\Big(\int_0^\infty \int_0^\infty |h_\lambda (tze^{-tz}(1-e^{-tz})h_\lambda (g))(x)|^2dx\frac{dt}{t}\Big)^{1/2}\\
&\leq C\Big(\int_0^\infty \int_0^\infty |tze^{-tz}(1-e^{-tz})|^2\frac{dt}{t}|h_\lambda (g)(z)|^2dz\Big)^{1/2}\\
&\leq C\|g\|_2.
\end{align*}
Hence, the sublinear operator $G_\lambda$ is bounded from $L^2(0,\infty )$ into itself. 

Let $I$ be a bounded interval in $(0,\infty )$. 

We now decompose the function $(i_d-P_{|I|}^\lambda )f$ as follows:
$$
(i_d-P_{|I|}^\lambda )f=\chi _{2I}(i_d-P_{|I|}^\lambda )f+\chi _{(0,\infty )\setminus 2I}(i_d-P_{|I|}^\lambda )f=:g_1+g_2.
$$
The arguments in \cite[p. 956]{DY} (see also \cite{DY2}) allow us to obtain
\begin{equation}\label{2.3}
\int_{\widehat{I}}|t(\partial _tP_t^\lambda )(i_d-P_t^\lambda )(g_1)(x)|^2\frac{dxdt}{t}\leq \|G_\lambda (g_1)\|_2^2\leq C\|g_1\|_2^2\leq C|I|\|f\|_{BMO(P^\lambda )}^2,
\end{equation}
and also, by using Lemma \ref{D4} and \eqref{D3}, that 
$$
|t(\partial _tP_t^\lambda )(i_d-P_t^\lambda )(g_2)(x)|\leq C\frac{t}{|I|}\|f\|_{BMO(P^\lambda )}.
$$
Then, we get
\begin{equation}\label{2.4}
\int_{\widehat{I}}|t(\partial _tP_t^\lambda )(i_d-P_t^\lambda )(g_2)(x)|^2\frac{dxdt}{t}\leq C|I|\|f\|_{BMO(P^\lambda )}^2.
\end{equation}
Inequality \eqref{2.1} follows now from \eqref{2.3} and \eqref{2.4}.

According again to Lemma \ref{D4}, \eqref{D3}  and by proceeding as in the bottom of \cite[p. 956]{DY} we obtain
$$
|t(\partial _tP_t^\lambda )(i_d-P_t^\lambda )P_{|I|}^\lambda (f)(x)|\leq C\frac{t}{|I|}\|f\|_{BMO(P^\lambda )},
$$
and \eqref{2.2} can be established.

If $F$ is a measurable function on $(0,\infty )\times (0,\infty )$ we define (see \cite[p. 488]{BCFR1})
$$
\Phi (F)(x)
:=\sup_{I\subset (0,\infty ), I \;{\rm bounded},x\in I}
\Big(\frac{1}{|I|}\int_0^{|I|}\int_I|F(y,t)|^2\frac{dydt}{t}\Big)^{1/2},\quad x\in (0,\infty ),
$$
and
$$
\Psi (F)(x)
:=\Big(\int_{\Gamma_+(x)}|F(y,t)|^2\frac{dydt}{t^2}\Big)^{1/2},\quad x\in (0,\infty ).
$$

Suppose that $b\in \mathcal{A}$ and consider $F(x,t):=t(\partial _tP_t^\lambda )(i_d-P_t^\lambda )f$ and $G(x,t):=t\partial _tP_t^\lambda (b)(x)$, $x,t\in (0, \infty )$. Since $\rho _f $ is a Carleson measure we have that $\Phi (F)\in L^\infty (0,\infty )$ and 
$$
\|\Phi (F)\|_\infty \leq C\|\rho _f\|_{\mathcal{C}}^{1/2}.
$$

On the other hand, from \cite[Proposition 4.1]{BCFR1}, we have that $\Psi (G)\in L^1(0,\infty )$ and $\|\Psi(G)\|_1\leq C\|b\|_{H_\lambda ^1(0,\infty )}$. Then, according to \cite[Proposition 4.3]{BCFR1} we get
\begin{align}\label{2.5}
&\int_0^\infty \int_0^\infty |t (\partial _tP_t^\lambda )(i_d-P_t^\lambda )f(y)||t\partial _tP_t^\lambda (b)(y)|\frac{dydt}{t}\nonumber\\
& \qquad \leq C\int_0^\infty \Phi (F)(y)\Psi (G)(y)dy \leq C\|f\|_{BMO(P^\lambda )}\|b\|_{H^1_\lambda (0,\infty )}.
\end{align}

It follows that
$$
\int_0^\infty \int_0^\infty t(\partial _tP_t^\lambda )(i_d-P_t^\lambda )f(y)t\partial _tP_t^\lambda (b)(y)\frac{dydt}{t}=\lim_{\varepsilon \rightarrow 0^+,N\rightarrow \infty }H(\varepsilon ,N),
$$
where, for every $0<\varepsilon <N<\infty$,
$$
H(\varepsilon ,N)
:=\int_\varepsilon ^N \int_0^\infty t(\partial _tP_t^\lambda )(i_d-P_t^\lambda )f(y)t\partial _tP_t^\lambda (b)(y)\frac{dydt}{t}.
$$
We can write 
$$
t(\partial _tP_t^\lambda )(i_d-P_t^\lambda )=t\partial _tP_t^\lambda -\frac{1}{2}t\partial _tP_{2t}^\lambda ,\quad t>0.
$$
Since $f\in L^1((0,\infty ), (1+x)^{-2}dx)$ and $b\in \mathcal{A}$, by using \cite[(4.8)]{BCFR1} we deduce that
\begin{align*}
&\int_0^\infty t(\partial _tP_t^\lambda )(i_d-P_t^\lambda )f(y)\partial _tP_t^\lambda (b)(y)dy\\
&\qquad \qquad =\int_0^\infty f(z)\int_0^\infty t\Big(\partial _tP_t^\lambda (y,z)-\frac{1}{2}\partial _tP_{2t}^\lambda (y,z)\Big)\partial _tP_t^\lambda (b)(y)dydz,\quad t>0,
\end{align*}
and, for every $0<\varepsilon<N<\infty$,
$$
H(\varepsilon ,N)=\int_0^\infty f(z)\int_\varepsilon ^N\int_0^\infty t\Big(\partial _tP_t^\lambda (y,z)-\frac{1}{2}\partial _tP_{2t}^\lambda (y,z)\Big)\partial _tP_t^\lambda (b)(y)dydtdz.
$$
According to \cite[(4.15)]{BCFR1} we obtain
$$
\lim_{\varepsilon \rightarrow 0^+, N\rightarrow \infty }\int_\varepsilon ^N\int_0^\infty 
t\partial _tP_t^\lambda (y,z)\partial _tP_t^\lambda (b)(y)dydt=\frac{b(z)}{4},\quad \mbox{ in }L^2(0,\infty ).
$$
In a similar way we can see that
$$
\lim_{\varepsilon \rightarrow 0^+, N\rightarrow \infty }\int_\varepsilon ^N\int_0^\infty 
t\partial _tP_{2t}^\lambda (y,z)\partial _tP_t^\lambda (b)(y)dydt=\frac{2b(z)}{9},\quad \mbox{ in }L^2(0,\infty ).
$$
Then,
$$
\lim_{\varepsilon \rightarrow 0^+, N\rightarrow \infty }\int_\varepsilon ^N\int_0^\infty
t\Big(\partial _tP_t^\lambda (y,z)-\frac{1}{2}\partial _tP_{2t}^\lambda (y,z)\Big)\partial _tP_t^\lambda (b)(y)dydt=\frac{5b(z)}{36},
$$
in $L^2(0,\infty )$.
By using now dominated convergence theorem as in \cite[p. 492]{BCFR1} we conclude that
\begin{equation}\label{2.6}
\int_0^\infty \int_0^\infty t\partial _t(P_t^\lambda )(i_d-P_t^\lambda )f(x)\partial _tP_t^\lambda (b)(x)\frac{dxdt}{t}=\frac{5}{36}\int_0^\infty f(x)b(x)dx.
\end{equation}
By combining \eqref{2.5} and \eqref{2.6} we get
$$
|T_f(b)|=\Big|\int_0^\infty f(x)b(x)dx\Big|
\leq C\|f\|_{BMO(P^\lambda )}\|b\|_{H_\lambda ^1(0,\infty )}.
$$
\end{proof}

\begin{proof}[Proof of Proposition \ref{Prop2.1}, $ii)$]
Assume that $T\in (H_\lambda ^1(0,\infty ))'$. There exists $f\in BMO_{\textrm{o}}(\mathbb{R})$ such that $Tg=T_fg$, for every $g\in \mathcal{A}$, and $\|f\|_{BMO_{\textrm{o}}(\mathbb{R})}\leq C\|T\|_{(H_\lambda ^1(0,\infty ))'}$ (\cite[p. 466]{BCFR1} and \cite[Theorem 1.10]{BDT}). We are going to see that $f\in BMO(P^\lambda )$.

Let $I$ be a bounded interval in $(0,\infty )$. We can write
\begin{align*}
\frac{1}{|I|}\int_I|f(x)-P_{|I|}^\lambda (f)(x)|dx&\leq \frac{1}{|I|}\int_I|f(x)-f_I|dx+\frac{1}{|I|}\int_I|P_{|I|}^\lambda (f-f_I)(x)|dx\\
&\quad +\frac{|f_I|}{|I|}\int_I|P_{|I|}^\lambda (1)(x)-1|dx=:J_1+J_2+J_3.
\end{align*}
Since $f\in BMO_{\textrm{o}}(\mathbb{R})$, $J_1\leq \|f\|_{BMO_{\textrm{o}}(\mathbb{R})}$. 
According to \eqref{D3} we have that
\begin{align}\label{3.6.1}
|P_{|I|}^\lambda (f-f_I)(x)|
&\leq C\int_0^\infty \frac{|I|}{|x-y|^2+|I|^2}|f(y)-f_I|dy \nonumber\\
& \leq C \Big(\int_{(0,\infty )\cap (x-|I|/2,x+|I|/2)}+\sum_{k\in \mathbb{N}}\int_{(0,\infty )\cap (B(x,2^k|I|)\setminus B(x,2^{k-1}|I|)}\Big)\frac{|I||f(y)-f_I|}{(x-y)^2+|I|^2}dy \nonumber\\
&  \leq C\Big(\frac{1}{|I|}\int_{(0,\infty )\cap (2I)}|f(y)-f_I|dy+\sum_{k\in \mathbb{N}}\frac{1}{2^{2k}|I|}\int_{(0,\infty )\cap 2^{k+2}I}|f(y)-f_I|dy\Big) \nonumber\\
& \leq C\|f\|_{BMO_{\textrm{o}}(\mathbb{R})}\Big(1+\sum_{k\in \mathbb{N}}\frac{k}{2^k}\Big)\leq C\|f\|_{BMO_{\textrm{o}}(\mathbb{R})},\quad x\in I.
\end{align}
Then, $J_2\leq C\|f\|_{BMO_{\textrm{o}}(\mathbb{R})}$.

In order to estimate $J_3$ we distinguish two cases. 
We consider firstly that $x_I\leq |I|$. According to (\ref{D3}) we get
\begin{equation}\label{dif1}
|P_{|I|}^\lambda (1)(x)-1|\leq C\int_0^\infty \frac{|I|}{(x-y)^2+|I|^2}dy +1\leq C,\quad x\in (0,\infty ).
\end{equation}
Then
\begin{equation}\label{J3}
J_3\leq C|f_I|\leq \frac{C}{|I|}\int_I|f(y)|dy\leq C\frac{x_I+|I|}{|I|(x_I+|I|)}\int_0^{x_I+|I|}|f(y)|dy\leq C\|f\|_{BMO_{\textrm{o}}(\mathbb{R})}.
\end{equation}

Suppose now that $I=(x_I-|I|/2, x_I+|I|/2)$, where $x_I>|I|$. We write
\begin{align*}
P_{|I|}^\lambda (1)(x)-1
&=\int_0^\infty P_{|I|}^\lambda (x,y)dy-\frac{1}{\pi}\int_{-\infty }^{+\infty }\frac{|I|}{(x-y)^2+|I|^2}dy\\
&=\Big(\int_0^{x/2} +\int_{2x}^\infty \Big) P_{|I|}^\lambda (x,y)dy
-\frac{1}{\pi }\Big(\int_{-\infty }^{x/2}+\int_{2x}^\infty \Big)\frac{|I|}{(x-y)^2+|I|^2}dy\\
&\quad +\int_{x/2}^{2x}\Big(P_{|I|}^\lambda (x,y)-\frac{1}{\pi }\frac{|I|}{(x-y)^2+|I|^2}\Big)dy\\
&=:\sum_{i=1}^3R_i(x),\quad x\in (0,\infty ).
\end{align*}
By using \eqref{D3} we get
\begin{align*}
|R_1(x)+R_2(x)|&
\leq C\left\{\Big(\int_0^{x/2}+\int _{2x}^\infty \Big)\frac{|I|}{(x-y)^2+|I|^2}dy+\int_0^\infty \frac{|I|}{(x+y)^2+|I|^2}dy \right\}\\
&\leq C|I|\left\{\Big(\int_0^{x/2}+\int_{2x}^\infty\Big)\frac{1}{(x-y)^2+|I|^2}dy +\int_{x/2}^{2x} \frac{1}{(x+y)^2+|I|^2}dy\right\}\\
&\leq C|I|\Big(\int_0^{x/2}\frac{dy}{x^2}+\int_{x/2}^\infty \frac{dy}{y^2}\Big)
\leq C\frac{|I|}{x}, \quad x\in (0,\infty ).
\end{align*}
To analyze $R_3(x)$, $x\in (0,\infty )$, we write
\begin{align*}
P_t^\lambda (x,y)
&=\frac{2\lambda t(xy)^\lambda }{\pi }\Big(\int_0^{\pi /2}+\int_{\pi /2}^\pi \Big)
\frac{(\sin \theta )^{2\lambda -1}}{((x-y)^2+t^2+2xy(1-\cos \theta ))^{\lambda +1}}d\theta\\
&=:P_t^{\lambda ,1}(x,y)+P_t^{\lambda ,2}(x,y),\quad x,y,t\in (0,\infty ),
\end{align*}
and 
\begin{align*}
P_t^{\lambda ,1}(x,y)
&=\frac{2\lambda t(xy)^\lambda }{\pi }\Big\{\int_0^{\pi /2}\frac{(\sin \theta )^{2\lambda -1}-\theta ^{2\lambda -1}}{((x-y)^2+t^2+2xy(1-\cos \theta))^{\lambda +1}}d\theta \\
&\quad +\int_0^{\pi /2}\Big(\frac{\theta ^{2\lambda -1}}{((x-y)^2+t^2+2xy(1-\cos \theta))^{\lambda +1}}-\frac{\theta ^{2\lambda -1}}{((x-y)^2+t^2+xy\theta^2)^{\lambda +1}}\Big)d\theta \\
&\quad +\int_0^\infty \frac{\theta ^{2\lambda -1}}{((x-y)^2+t^2+xy\theta ^2)^{\lambda +1}}d\theta -\int_{\pi /2}^\infty \frac{\theta ^{2\lambda -1}}{((x-y)^2+t^2+xy\theta ^2)^{\lambda +1}}d\theta \Big\}\\
&=:\frac{2\lambda t(xy)^\lambda }{\pi }\sum_{i=1}^4I_i(x,y,t), \quad x,y,t\in (0,\infty ).
\end{align*}
We observe that
$$
\frac{2\lambda t(xy)^\lambda }{\pi }I_3(x,y,t)=\frac{1}{\pi}\frac{t}{(x-y)^2+t^2}=P_t(x-y),\quad x,y,t\in (0,\infty ),
$$
and then, we obtain, for each $x,y,t\in (0,\infty )$,
\begin{equation}\label{decomposition}
P_t^\lambda (x,y)-P_t(x-y)=\frac{2\lambda t(xy)^\lambda }{\pi }(I_1(x,y,t)+I_2(x,y,t)+I_4(x,y,t))+P_t^{\lambda ,2}(x,y).
\end{equation}

We have that, 
$$
|P_t^{\lambda,2}(x,y)|\leq C\frac{t(xy)^\lambda }{(x^2+y^2+t^2)^{\lambda +1}}\leq C\frac{t}{x^2},\quad x,y,t\in (0,\infty).
$$

By using mean value theorem we get
\begin{align*}
|I_1(x,y,t)|&\leq C\int_0^{\pi /2}\frac{\theta ^{2\lambda +1}}{((x-y)^2+t^2+xy\theta ^2)^{\lambda +1}}d\theta\leq C\int_0^{\pi /2}\frac{\theta ^{2\lambda +1}}{(|x-y|+t+x\theta )^{2\lambda +2}}d\theta\\
&\leq 
\frac{C}{x^{2\lambda +3/2}|x-y|^{1/2}},\quad 0<\frac{x}{2}<y<2x,\;t>0,
\end{align*}
and
\begin{align*}
|I_2(x,y,t)|&\leq C\int_0^{\pi /2}\frac{xy\theta ^{2\lambda +3}}{((x-y)^2+t^2+xy\theta ^2)^{\lambda +2}}d\theta \leq C\int_0^{\pi /2}\frac{x^2\theta ^{2\lambda +3}}{(|x-y|+t+x\theta )^{2\lambda +4}}d\theta  \\
&\leq \frac{C}{x^{2\lambda +3/2}|x-y|^{1/2}},\quad 0<\frac{x}{2}<y<2x,\;t>0.
\end{align*}

Also, for every $t>0$ and $0<\frac{x}{2}<y<2x$, we can write
\begin{align*}
|I_4(x,y,t)|&=\frac{(xy)^{-\lambda}}{(x-y)^2+t^2}\int_{\frac{\pi}{2}\sqrt{xy/((x-y)^2+t^2)}}^\infty \frac{u^{2\lambda -1}}{(1+u^2)^{\lambda +1}}du\\
&\leq C\frac{x^{-2\lambda}}{(x-y)^2+t^2}
\int_{\frac{\pi}{2}\sqrt{xy/((x-y)^2+t^2)}}^\infty \frac{du}{u^3}\leq \frac{C}{x^{2\lambda +2}}.
\end{align*}

From \eqref{decomposition} and by putting together the above estimates we get
$$
\left|P_t^\lambda (x,y)-P_t(x-y)\right|
\leq Ct\Big(\frac{1}{x^{3/2}|x-y|^{1/2}}+\frac{1}{x^2}\Big),\quad 0<\frac{x}{2}<y<2x,\;t>0.
$$
It follows that
$$
|R_3(x)|
\leq C|I|\int_{x/2}^{2x}\Big(\frac{1}{x^{3/2}|x-y|^{1/2}}+\frac{1}{x^2}\Big)dy
\leq C\frac{|I|}{x},\quad x\in (0,\infty ).
$$
We obtain that
\begin{equation}\label{dif2}
|P_{|I|}^\lambda (1)(x)-1|\leq C\frac{|I|}{x},\quad x\in (0,\infty ).
\end{equation}
Since $x_I>|I|$, then $x_I-|I|/2>2^{-1}x_I$, and we get
\begin{align}\label{J3b}
J_3&\leq C|f_I|\int_I\frac{1}{x}dx\leq \frac{C}{|I|}\int_I|f(y)|dy\frac{|I|}{x_I-|I|/2}\nonumber\\
&\leq C\frac{x_I+|I|}{(x_I-|I|/2)(x_I+|I|)}\int_0^{x_I+|I|}|f(y)|dy\leq C\|f\|_{BMO_{\textrm{o}}(\mathbb{R})}.
\end{align}
We conclude that
$$
\frac{1}{|I|}\int_I|f(x)-P_{|I|}^\lambda (f)(x)|dx\leq C\|f\|_{BMO_{\textrm{o}}(\mathbb{R})}.
$$
Thus we prove that $f\in BMO(P^\lambda )$ and that
$$
\|f\|_{BMO(P^\lambda )}\leq C\|f\|_{BMO_{\textrm{o}}(\mathbb{R})}\leq C\|T\|_{(H_1^\lambda (0,\infty ))'}.
$$
\end{proof}

From Proposition \ref{Prop2.1} and since $BMO_{\textrm{o}}(\mathbb{R})$ is the dual space of $H_\lambda ^1(0,\infty )$ (\cite[p. 466]{BCFR1}),  we can deduce the equality of $BMO(P^\lambda )$ and $BMO_{\textrm{o}}(\mathbb{R})$.

\begin{cor}\label{cor2.1}
Let $\lambda >0$. Then, $BMO(P^\lambda )=BMO_{\textrm{o}}(\mathbb{R})$ algebraic and topologically.
\end{cor}

The following property will be very useful in the sequel.

\begin{prop}\label{Prop2.2}
Let $\lambda >0$. There exists $C>0$ such that, for every $f\in BMO_{\textrm{o}}(\mathbb{R})$,
$$
|t\partial _tP_t^\lambda (f)(x)|+|tD_{\lambda ,x}P_t^\lambda (f)(x)|\leq C\|f\|_{BMO_{\textrm{o}}(\mathbb{R})},\quad x,t\in (0,\infty ).
$$
\end{prop}
\begin{proof}
Let $f\in BMO_{\textrm{o}}(\mathbb{R})$ and consider the odd extension $f_{\rm{o}}$ of $f$ to $\mathbb{R}$. We have that $f_{\rm{o}}\in BMO(\mathbb{R})$ and  
\begin{align*}
P_t(f_{\rm o})(x)&=\int_0^\infty (P_t(x-y)-P_t(x+y))f(y)dy,\quad x,t\in (0,\infty ).
\end{align*}

Since $f_{\rm o}\in BMO(\mathbb{R})$ and $\|f_{\rm o}\|_{BMO(\mathbb{R})}\leq C\|f\|_{BMO_{\rm o}(\mathbb{R})}$, from \cite[(2), p. 22]{Wi} (see also \cite{Ga}) we deduce that
$$
|t\partial _tP_t(f_{\rm o})(x)|+|t\partial_xP_t(f_{\rm o})(x)| \leq C\|f\|_{BMO_{\rm o}(\mathbb{R})},\quad t,x\in (0,\infty ).
$$

Also, we can write
\begin{align}\label{BesselClassical}
t\partial _tP_t^\lambda (f)(x)
&=t\partial _tP_t(f_{\rm o})(x)+\int_0^\infty t\partial _t(P_t^\lambda (x,y)-P_t(x-y)+P_t(x+y))f(y)dy\nonumber\\
&=t\partial _tP_t(f_{\rm o})(x)
+\Big(\int_0^{x/2}+\int_{2x}^\infty\Big)t\partial _tP_t^\lambda (x,y)f(y)dy \nonumber\\
&\quad +\Big(\int_0^{x/2}+\int_{2x}^\infty\Big)t\partial _t(P_t(x+y)-P_t(x-y))f(y)dy\nonumber\\
&\quad +\int_{x/2}^{2x}t\partial _tP_t(x+y)f(y)dy+\int_{x/2}^{2x}t\partial _t(P_t^\lambda (x,y)-P_t(x-y))f(y)dy\nonumber\\
&=t\partial _tP_t(f_{\rm o})(x)+\sum_{i=1}^4 J_i(x,t),\quad x,t\in (0,\infty ).
\end{align}

By using Lemma \ref{D4} and (\ref{D2}) we deduce that
\begin{align}\label{J1}
|J_1(x,t)|
&\leq C\Big(\int_0^{x/2}+\int_{2x}^\infty \Big)
\frac{t(xy)^\lambda |f(y)|}{((x-y)^2+t^2)^{\lambda +1}}dy\nonumber\\
&\leq C\Big(\int_0^{x/2}\frac{tx^{2\lambda}|f(y)|}{(x^2+t^2)^{\lambda +1}}dy+\int_{2x}^\infty \frac{t(xy)^\lambda|f(y)|}{(y^2+t^2)^{\lambda +1}}dy\Big)\nonumber\\
&\leq C\Big(\frac{1}{x}\int_0^x|f(y)|dy+x^\lambda\int_{2x}^\infty \frac{|f(y)|}{y^{\lambda +1}}dy\Big) \nonumber\\
&\leq C\Big(\|f\|_{BMO_{\textrm{o}}(\mathbb{R})}+x^\lambda\sum_{k=1}^\infty \int_{2^kx}^{2^{k+1}x}\frac{|f(y)|}{y^{\lambda +1}}dy\Big)\nonumber\\
&\leq  C\Big(\|f\|_{BMO_{\textrm{o}}(\mathbb{R})}+x^\lambda \sum_{k=1}^\infty \frac{1}{(2^kx)^{\lambda +1}}\int_0^{2^{k+1}x}|f(y)|dy\Big)\nonumber\\
&\leq  C\Big(\|f\|_{BMO_{\textrm{o}}(\mathbb{R})}+\sum_{k=1}^\infty 2^{-k\lambda }\frac{1}{2^{k+1}x}\int_0^{2^{k+1}x}|f(y)|dy\Big)\nonumber\\
&\leq C\|f\|_{BMO_{\rm{o}}(\mathbb{R})},\quad x,t\in (0,\infty ). 
\end{align}

On the other hand, 
\begin{align*}
|J_2(x,t)|
&=\frac{1}{\pi}\Big|\Big(\int_0^{x/2}+\int_{2x}^\infty\Big)t\partial _t\Big(\frac{4xyt}{((x-y)^2+t^2)((x+y)^2+t^2)}\Big)f(y)dy\Big|\\
&=\frac{1}{\pi}\Big|\Big(\int_0^{x/2}+\int_{2x}^\infty\Big)\frac{4xyt}{((x-y)^2+t^2)((x+y)^2+t^2)} \\
& \qquad  \qquad \qquad  \qquad -\frac{8xyt^3((x+y)^2+(x-y)^2+2t^2)}{((x-y)^2+t^2)^2((x+y)^2+t^2)^2}f(y)dy\Big|\\
&\leq C\Big(\int_0^{x/2}+\int_{2x}^\infty\Big)\frac{xyt}{((x-y)^2+t^2)((x+y)^2+t^2)}|f(y)|dy\\
&\leq C\Big(\frac{tx^2}{(x^2+t^2)^2}\int_0^x|f(y)|dy+tx\int_{2x}^\infty \frac{y|f(y)|}{(y^2+t^2)^2}dy\Big)\\
&\leq C\Big(\frac{1}{x}\int_0^x|f(y)|dy+x\int_{2x}^\infty \frac{|f(y)|}{y^2}dy\Big) 
\leq C\|f\|_{BMO_{\rm{o}}(\mathbb{R})},\quad x,t\in (0,\infty ).
\end{align*}
The last inequality is obtained by proceeding as in \eqref{J1} for $\lambda =1$.

Also, we have that
$$
|J_3(x,t)|
=\frac{1}{\pi}\Big|\int_{x/2}^{2x}\Big(\frac{t}{(x+y)^2+t^2}-\frac{2t^3}{((x+y)^2+t^2)^2}\Big)f(y)dy\Big|
\leq \frac{C}{x}\int_{x/2}^{2x}|f(y)|dy
\leq C\|f\|_{BMO_{\rm o}(\mathbb{R})}.
$$

Finally, in order to estimate $J_4(x,t)$, $x,t\in (0,\infty )$, we consider \eqref{decomposition} and write, for each $x,y,t\in (0,\infty )$, 
\begin{equation}\label{diferencia}
t\partial _t(P_t^\lambda (x,y)-P_t(x-y))=\frac{2\lambda }{\pi}(xy)^\lambda\sum_{i=1,2,4}t\partial _t(tI_i(x,y,t))+t\partial _tP_t^{\lambda ,2}(x,y).
\end{equation}

By using the mean value theorem we get
\begin{align*}
|t\partial _t(tI_1(x,y,t))|
&=\Big|tI_1(x,y,t)-2(\lambda +1)t^3\int_0^{\pi /2}\frac{(\sin \theta )^{2\lambda -1}-\theta ^{2\lambda -1}}{((x-y)^2+t^2+2xy(1-\cos \theta ))^{\lambda +2}}d\theta\Big|\\
&\leq Ct\int_0^{\pi /2}\frac{|(\sin \theta )^{2\lambda -1}-\theta ^{2\lambda -1}|}{((x-y)^2+t^2+2xy(1-\cos \theta ))^{\lambda +1}}d\theta\\
&\leq Ct\int_0^{\pi /2}\frac{\theta ^{2\lambda +1}}{(|x-y|+t+x\theta )^{2\lambda +2}}d\theta \leq \frac{C}{x^{2\lambda +1}},\quad 0<\frac{x}{2}<y<2x,\;t>0, 
\end{align*}
and
\begin{align*}
|t\partial _t(tI_2(x,y,t))|
&=\Big|tI_2(x,y,t)-2(\lambda +1)t^3\int_0^{\pi /2}\Big[\frac{\theta ^{2\lambda -1}}{((x-y)^2+t^2+2xy(1-\cos \theta ))^{\lambda +2}} \\
& \qquad \qquad \qquad \qquad \qquad \qquad \qquad - \frac{\theta ^{2\lambda -1}}{((x-y)^2+t^2+xy\theta ^2)^{\lambda +2}}\Big]d\theta \Big| \\
& \leq Cxy\Big(\int_0^{\pi /2}\frac{t\theta ^{2\lambda +3}}{((x-y)^2+t^2+xy\theta ^2)^{\lambda +2}}d\theta +\int_0^{\pi /2}\frac{t^3\theta ^{2\lambda +3}}{((x-y)^2+t^2+xy\theta ^2)^{\lambda +3}}d\theta \Big)\\
& \leq Ctx^2\int_0^{\pi /2}\frac{\theta ^{2\lambda +3}}{(|x-y|+t+x\theta )^{2\lambda +4}}d\theta 
\leq \frac{C}{x^{2\lambda +1}},
\quad 0<\frac{x}{2}<y<2x,\;t>0.
\end{align*}
Also, we obtain, when $0<x/2<y<2x$ and $t>0$,
\begin{align*}
|t\partial _t(tI_4(x,y,t))|
&=\Big|tI_4(x,y,t)-2(\lambda +1)t^3\int_{\pi /2}^\infty \frac{\theta ^{2\lambda -1}}{((x-y)^2+t^2+xy\theta ^2)^{\lambda +2}}d\theta \Big|\\
&\leq Ct\int_{\pi /2}^\infty \frac{\theta ^{2\lambda -1}}{((x-y)^2+t^2+xy\theta ^2)^{\lambda +1}}d\theta\leq \frac{C}{x^{2\lambda +1}}\int_{\pi /2}^\infty \frac{d\theta}{\theta ^2}\leq \frac{C}{x^{2\lambda +1}}.
\end{align*}
Finally, we can write
\begin{align*}
|t\partial _t(P_t^{\lambda ,2}(x,y))|
&=\Big|P_t^{\lambda ,2}(x,y)-\frac{4\lambda (\lambda +1)}{\pi}t^3
(xy)^\lambda \int_{\pi /2}^\pi\frac{(\sin \theta )^{2\lambda -1}}{((x-y)^2+t^2+2xy(1-\cos \theta))^{\lambda +2}}d\theta \Big|\\
&\leq Ctx^{2\lambda }\int_{\pi /2}^\pi \frac{\theta ^{2\lambda -1}}{(x+y+t)^{2\lambda +2}}d\theta\leq \frac{C}{x}, \quad 0<\frac{x}{2}<y<2x,\;t>0..
\end{align*}

From \eqref{diferencia} and by combining the above estimates, it follows that
$$
J_4(x,t)\leq \frac{C}{x}\int_0^{2x}|f(y)|dy\leq C\|f\|_{BMO_{\rm o}(\mathbb{R})},\quad x,t\in (0,\infty ).
$$

Equality \eqref{BesselClassical} allows us to conclude that
$$
|t\partial _tP_t^\lambda (f)(x)|\leq C\|f\|_{BMO_{\rm o}(\mathbb{R})},\quad x,t \in (0,\infty ).
$$

We are going to see now that
\begin{equation*}%\label{2.8}
|tD_{\lambda ,x}P_t^\lambda(f)(x)|\leq C\|f\|_{BMO_{\rm o}(\mathbb{R})},\quad x,t\in (0,\infty ).
\end{equation*}

We use a decomposition similar to \eqref{BesselClassical} and write
\begin{align}\label{BesselClassicalD}
tD_{\lambda ,x}P_t^\lambda (f)(x)
&=t\partial _xP_t(f_{\rm o})(x)+t\int_0^\infty [D_{\lambda ,x}(P_t^\lambda (x,y))-\partial _x(P_t(x-y)+P_t(x+y))]f(y)dy\nonumber\\
&=t\partial _xP_t(f_{\rm o})(x)
+\Big(\int_0^{x/2}+\int_{2x}^\infty\Big)tD_{\lambda ,x}(P_t^\lambda (x,y))f(y)dy\nonumber\\
&\quad +\Big(\int_0^{x/2}+\int_{2x}^\infty\Big)t\partial _x(P_t(x+y)-P_t(x-y))f(y)dy\nonumber\\
&\quad +\int_{x/2}^{2x}t\partial _xP_t(x+y)f(y)dy+\int_{x/2}^{2x}t[D_{\lambda ,x}(P_t^\lambda (x,y))-\partial _x(P_t(x-y))]f(y)dy\nonumber\\
&=:t\partial _xP_t(f_{\rm o})(x)+\sum_{i=1}^4 H_i(x,t),\quad x,t\in (0,\infty ).
\end{align}

By using Lemma \ref{D4} and (\ref{D2}) in the same way as in \eqref{J1} we obtain
$$
|H_1(x,t)|\leq C\|f\|_{BMO_{\rm o}(\mathbb{R})},\quad x,t\in (0,\infty ).
$$

On the other hand,
\begin{align*}
|H_2(x,t)|
&=\frac{1}{\pi}\Big|\Big(\int_0^{x/2}+\int_{2x}^\infty\Big)\partial _x\Big(\frac{4xyt^2
}{((x-y)^2+t^2)((x+y)^2+t^2)}\Big)f(y)dy\Big|\\
&\hspace{-1cm}=\frac{1}{\pi}\Big|\Big(\int_0^{x/2}+\int_{2x}^\infty\Big)\frac{4yt^2}{((x-y)^2+t^2)((x+y)^2+t^2)}\\
&\quad -4xyt^2\frac{2(x-y)((x+y)^2+t^2)+2(x+y)((x-y)^2+t^2)}{((x-y)^2+t^2)^2((x+y)^2+t^2)^2}f(y)dy\Big|\\
&\hspace{-1cm}\leq C\Big(\int_0^{x/2}+\int_{2x}^\infty\Big)\frac{yt^2}{((x-y)^2+t^2)((x+y)^2+t^2)}\Big(1+\frac{x|x-y|}{(|x-y|+t)^2}+ \frac{x(x+y)}{(x+y+t)^2}\Big)|f(y)|dy\\
&\hspace{-1cm}\leq C\Big(\frac{xt^2}{(x^2+t^2)^2}\int_0^x|f(y)|dy+\int_{2x}^\infty \frac{yt^2|f(y)|}{(y^2+t^2)(x+y+t)^2}dy\Big)\\
&\hspace{-1cm} \leq C\Big(\frac{1}{x}\int_0^x|f(y)|dy+\int_{2x}^\infty \frac{t|f(y)|}{(x+y+t)^2}dy\Big)\\
&\hspace{-1cm} \leq C\Big(\|f\|_{BMO_{\rm{o}}(\mathbb{R})}
+ \Big(\int_{2x}^{2x+t}+\sum_{k=0}^\infty \int_{2x+2^kt}^{2x+2^{k+1}t}\Big)\frac{t}{(x+y+t)^2}|f(y)|dy\Big)\\
&\hspace{-1cm} \leq C\Big(\|f\|_{BMO_{\rm{o}}(\mathbb{R})}+\frac{1}{2x+t}\int_0^{2x+t}|f(y)|dy +\sum_{k=0}^\infty 2^{-k}\frac{1}{2x+2^{k+1}t}\int_0^{2x+2^{k+1}t}|f(y)|dy\Big)\\
&\hspace{-1cm} \leq C\|f\|_{BMO_{\rm o}(\mathbb{R})},\quad x,t\in (0,\infty ).
\end{align*}

Also, it follows that
$$
|H_3(x,t)|
=\frac{2}{\pi}\Big|\int_{x/2}^{2x}\frac{(x+y)t^2}{((x+y)^2+t^2)^2}f(y)dy\Big|
\leq \frac{C}{x}\int_0^{2x}|f(y)|dy\leq C\|f\|_{BMO_{\rm o}(\mathbb{R})}.
$$

We deal now with $H_4(x,t)$, $x,t\in (0,\infty )$. From \eqref{decomposition} we have that
\begin{align}\label{Dlambda}
D_{\lambda ,x}P_t^\lambda (x,y)-\partial _xP_t(x-y)&=\partial _x[P_t^\lambda (x,y)-P_t(x-y)]-\frac{\lambda}{x}P_t^\lambda (x,y)\nonumber\\
&\hspace{-2cm}=\frac{2\lambda}{\pi}\sum_{i=1,2,4}t\partial_x[(xy)^\lambda I_i(x,y,t)]+\partial_xP_t^{\lambda , 2}(x,y)-\frac{\lambda}{x}P_t^\lambda (x,y),\quad x,y,t\in (0,\infty ).
\end{align}

Again by using the mean value theorem we get
\begin{align*}
|t\partial _x[(xy)^\lambda I_1(x,y,t)]|
&=\Big| \frac{\lambda t}{x} (xy)^\lambda I_1(x,y,t)\\
& \qquad -2(\lambda +1)t(xy)^\lambda\int_0^{\pi /2}\frac{[(\sin \theta )^{2\lambda -1}-\theta ^{2\lambda -1}][(x-y)+y(1-\cos \theta)]}{((x-y)^2+t^2+2xy(1-\cos \theta ))^{\lambda +2}}d\theta\Big|\\
&\leq Ctx^{2\lambda} \Big(\frac{1}{x}\int_0^{\pi /2}\frac{\theta ^{2\lambda +1}}{(|x-y|+t+x\theta )^{2\lambda +2}}d\theta+\int_0^{\pi /2}\frac{\theta ^{2\lambda +1}(|x-y|+y\theta ^2)}{(|x-y|+t+x\theta )^{2\lambda +4}}d\theta\Big)\\
&\leq \frac{C}{tx}\Big(\int_0^{\pi /2}\theta d\theta +\int_0^{\pi/2}(1+\theta )d\theta \Big)
\leq \frac{C}{tx},\quad 0<\frac{x}{2}<y<2x,\;t>0, 
\end{align*}
and
\begin{align*}
|t\partial _x[(xy)^\lambda I_2(x,y,t)]|
&=\Big|\frac{\lambda t}{x} (xy)^\lambda I_2(x,y,t)\\
&\quad -2(\lambda +1)t(xy)^\lambda\int_0^{\pi /2}\theta ^{2\lambda -1}\Big[\frac{(x-y)+y(1-\cos \theta)}{((x-y)^2+t^2+2xy(1-\cos \theta ))^{\lambda +2}} \\
& \qquad \qquad \qquad \qquad \qquad \qquad \qquad - \frac{(x-y)+y\theta ^2}{((x-y)^2+t^2+xy\theta ^2)^{\lambda +2}} \Big]d\theta\Big| \\
&\leq Ct(xy)^\lambda \Big\{\frac{1}{x}\int_0^{\pi /2}\frac{\theta ^{2\lambda +3}xy}{(|x-y|+t+x\theta )^{2\lambda +4}}d\theta \\
&\quad  +\int_0^{\pi /2}\theta ^{2\lambda +3}\Big(\frac{y}{(|x-y|+t+x\theta )^{2\lambda +4}}+\frac{xy(|x-y|+y\theta^2 )}{(|x-y|+t+x\theta )^{2\lambda +6}}\Big)d\theta \Big\}\\
&\leq C\Big(tx^{2\lambda +1}\int_0^{\pi /2}\frac{\theta ^{2\lambda +2}}{(|x-y|+t+x\theta )^{2\lambda +4}}d\theta \Big)
\leq \frac{C}{tx},\quad 0<\frac{x}{2}<y<2x,\;t>0.
\end{align*}

Now, we write
\begin{align*}
|t\partial _x((xy)^\lambda I_4(x,y,t))|&=\Big|\frac{\lambda t}{x} (xy)^\lambda I_4(x,y,t)\\
&\quad -2(\lambda +1)t(xy)^\lambda \int_{\pi /2}^\infty\frac{\theta ^{2\lambda -1}(x-y+y\theta ^2)}{((x-y)^2+t^2+xy\theta ^2)^{\lambda +2}}d\theta \Big|\\
&\leq Ctx^{2\lambda -1}\int_{\pi /2}^\infty \frac{\theta ^{2\lambda -1}}{(|x-y|+t+x\theta )^{2\lambda +2}}d\theta \\
&=C\frac{t}{x(|x-y|+t)^2}\int_{\frac{\pi x}{2(|x-y|+t)}}^\infty \frac{u^{2\lambda -1}}{(1+u)^{2\lambda +2}}du\leq \frac{C}{tx}, \quad 0<\frac{x}{2}<y<2x,\;t>0.
\end{align*}
Finally, it is clear from \eqref{D3} that
$$
\frac{1}{x}|P_t^\lambda (x,y)|\leq \frac{C}{tx},\quad x,t\in (0,\infty ),
$$
and, also, 
\begin{align*}
|\partial _x(P_t^{\lambda ,2}(x,y))|
&=\Big|\frac{\lambda }{x}P_t^{\lambda ,2}(x,y)-\frac{4\lambda(\lambda +1)}{\pi}t(xy)^\lambda \int_{\pi /2}^\pi\frac{(\sin \theta )^{2\lambda -1}[x-y+y(1-\cos \theta)]}{((x-y)^2+t^2+2xy(1-\cos \theta))^{\lambda +2}}d\theta \Big|\\
&\leq Ctx^{2\lambda -1}\int_{\pi /2}^\pi \frac{\theta ^{2\lambda -1}}{(x+y+t)^{2\lambda +2}}d\theta\leq \frac{C}{tx}, \quad 0<\frac{x}{2}<y<2x,\;t>0.
\end{align*}
By combining the above estimates and taking into account \eqref{Dlambda} it follows that
$$
H_4(x,t)\leq \frac{C}{x}\int_0^{2x}|f(y)|dy\leq C\|f\|_{BMO_{\rm o}(\mathbb{R})},\quad x,t\in (0,\infty ).
$$

Then, from \eqref{BesselClassicalD} we conclude that
$$
|D_{\lambda ,x}P_t^\lambda (f)(x)|\leq C\|f\|_{BMO_{\rm o}(\mathbb{R})},\quad x,t\in (0,\infty ).
$$
\end{proof}

%\newpage
%%% --------------------------------------------------------
\section{Proof of Theorem \ref{Th1.3}}\label{Sect4}

\begin{proof}[Proof of Theorem \ref{Th1.3}, $(i)$]
Suppose that $\mu$ is a Carleson measure on $(0,\infty)\times (0,\infty )$. According to Corollary \ref{cor2.1} in order to see that $S_{\mu,P^\lambda}\in BMO_{\rm o}(\mathbb R)$ it is sufficient to see that there exists $C>0$ such that, for every bounded interval $I\subset (0,\infty)$,

\begin{equation}\label{C.1}
\frac{1}{|I|}\int_I|S_{\mu,P^\lambda}(x)-P_{|I|}^\lambda(S_{\mu,P^\lambda})(x)|dx\leq C\|\mu\|_{\mathcal C}.
\end{equation}
We proceed as in the proof of \cite[Proposition 2.5]{CDLSY}.

Let $I$ be a bounded interval in $(0,\infty)$. We can write
\begin{align*}
&  \int_I  |S_{\mu,P^\lambda}(x)-P_{|I|}^\lambda(S_{\mu,P^\lambda})(x)|dx \\
& \qquad \leq C \int_I\int_{(0,\infty)^2}|P_t^\lambda(x,y)-P_{t+|I|}^\lambda(x,y)|d\mu(y,t)dx \\
& \qquad  \leq\Big(\int_I\int_{\widehat{2I}}+\int_I\int_{(0,\infty)^2\setminus \widehat{2I}}\Big)|P_t^\lambda(x,y)-P_{t+|I|}^\lambda(x,y)|d\mu(y,t)dx \\
& \qquad  =:I_1+I_2.
\end{align*}
According to (\ref{D3}), since $\mu$ is a Carleson measure on $(0,\infty)\times (0,\infty )$, we get
\begin{align*}
  I_1\leq C 
  & \int_{\widehat{2I}}\int_I\Big(\frac{t}{(x-y)^2+t^2}+ \frac{t+|I|}{(x-y)^2+(t+|I|)^2}\Big)dxd\mu(y,t) 
  \leq C\mu(\widehat{2I})\leq C|I|\|\mu\|_{\mathcal C}.
\end{align*}
Also by Lemmma \ref{D4} and (\ref{D3}), we obtain
\begin{align*}
 I_2
 &\leq  C \int_I\int_{(0,\infty)^2\setminus \widehat{2I}}\int_0^{|I|}|\partial_sP_{t+s}^\lambda(x,y)|dsd\mu(y,t)dx \\
& \leq C\sum_{k=1}^\infty\int_I\int_0^{|I|}\int_{\widehat{2^{k+1}I}\setminus\widehat{2^{k}I}}\frac{1}{(x-y)^2+(s+t)^2}d\mu(y,t)dsdx \\
& \leq C\sum_{k=1}^\infty\mu(\widehat{2^{k+1}I})\frac{1}{(2^{k+1}|I|)^2}|I|^2\leq C|I|\|\mu\|_{\mathcal C}.
\end{align*}
Thus, (\ref{C.1}) is proved.
\end{proof}

\begin{proof}[Proof of Theorem \ref{Th1.3}, $(ii)$]
We will use the procedure developed by Wilson (\cite{Wi}) (see also \cite{CDLSY}). We need to make modifications and to justify each step in our setting.

Let $Q$ be a bounded interval in $(0,\infty )$. In what follows we consider right-open intervals and denote by $x_Q$ the center of $Q$, and by $t_Q$ the length of $Q$.

Assume that $f\in BMO_{\rm o}(\mathbb{R})$ with $\mbox{supp }f\subset (0,1)$. We consider $u(x,t):=P_t^\lambda (f)(x)$, $x,t\in (0,\infty )$, and take $Q_0:=[0,2)$. In what follows we consider right-open intervals.

We now construct the $k$-th generation of subintervals of $Q_0$ as follows. By $A$ we denote a positive constant that will be fixed later. The 0-th generation is defined by $G_0:=\{Q_0\}$. 
For every $k \in \N$, the $(k+1)$-th generation $G_{k+1}$ is defined recursively as follows. A dyadic interval $Q\subset Q_0$ is in $G_{k+1}$ when
\begin{itemize}
\item[$(a)$] there exists $Q_1\in G_k$ such that $Q\subset Q_1$,

\item[$(b)$] $Q$ is a maximal dyadic with respect to the property 
$$|x_Q^{-\lambda }u(x_Q,t_Q)-x_{Q_1}^{-\lambda }u(x_{Q_1},t_{Q_1})|>Ax_Q^{-\lambda}.$$
\end{itemize}
Note that the properties of the dyadic intervals and the maximal property ($b$) imply that, if $k\in \mathbb{N}$ and $Q_1,Q_2\in G_k$, then $Q_1=Q_2$ or $Q_1\cap Q_2=\emptyset$.

For every $k\in \mathbb{N}$ and $Q\in G_k$ we define the set 
$$\Sigma_Q
:=\widehat{Q}\setminus \bigcup_{Q'\subset Q, \, Q'\in G_{k+1}}\widehat{Q'}.$$
In the following figure where a possible $\Sigma _Q$ is represented, the dark grey squares are the Carleson boxes of those cubes $Q'\subset Q$ that belong to $G_{k+1}$.  

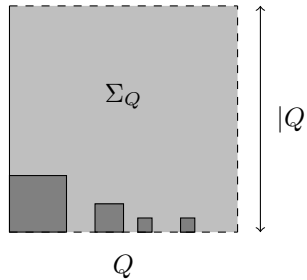
\begin{figure}[h!]
   \begin{tikzpicture}[scale=1.5]
        \draw[<->] (2.2,0) -- (2.2,2);      
		\node at (2.5,1) {$|Q|$};
		\node at (1,-0.3) {$Q$};
	     
		\path[fill=lightgray](0,0.5)--(0.5,0.5)--(0.5,0)--(0.75,0)--(0.75,0.25)--(0.5,0.25)--(1,0.25)--(1,0)--(1.125,0)--(1.125,0.125)--(1.25,0.125)--(1.25,0)--(1.5,0)--(1.5,0.125)--(1.625,0.125)--(1.625,0)--(2,0)--(2,2)--(0,2);
		\draw[fill=gray](0,0) rectangle (0.5,0.5);
		\draw[fill=gray](0.75,0) rectangle (1,0.25);
		\draw[fill=gray](1.125,0) rectangle (1.25,0.125);
		\draw[fill=gray](1.5,0) rectangle (1.625,0.125);
		\node at (1,1.2){$\Sigma _Q$};	
		
		\draw[-] (0,0) -- (0,2) ; %tienda de Q 
 		\draw[dashed] (0,2) -- (2,2);
         \draw[dashed] (0,0) -- (2,0);            
         \draw[dashed] (2,0) -- (2,2);
 \end{tikzpicture}
%\caption{Global regions}\label{gr}
%\label{fig:regions}
\end{figure}

\newpage
The set $\Sigma_Q$, $Q\in G_k$, $k\in \mathbb{N}$, can be written in a different and useful way. For every interval $J$ we define $T(J)$ as follows
$$
T(J)
:=\Big\{(x,t): x\in J\mbox{ and } \frac{\ell (J)}{2}\leq t<  \ell (J)\Big\}.
$$
It is clear that, for every dyadic interval $S\subset (0,\infty)$, we have that
$$
\widehat{S}
=\bigcup_{J\subset S, \, J \;{\rm dyadic }}T(J).
$$
Then, for every $k\in \mathbb{N}$ and $Q\in G_k$, we have that 
$$
\Sigma_Q=\bigcup _{J\in \mathcal{A}(Q)}T(J),
$$
where $\mathcal{A}(Q):=\{J\mbox{ dyadic}: J\subseteq Q, J\cap S^{\rm c}\not=\emptyset, \mbox{ for every }S\in G_{k+1}\}$. 

Now, take $k\in \mathbb{N}$ and $Q\in G_k$. We are going to see that
\begin{equation}\label{3.1}
|x^{-\lambda }u(x,t)-x_Q^{-\lambda }u(x_Q,t_Q)|\leq Cx^{-\lambda }(A+\|f\|_{BMO_{\rm o}(\mathbb{R})}),\quad (x,t)\in \Sigma _Q.
\end{equation}
Here $C>0$ does not depend on $k$ or $Q$.

Suppose that $(x,t)\in \Sigma _Q$. There exists $J\in \mathcal{A}(Q)$ such that $(x,t)\in T(J)$.
According to the definition of $G_{k+1}$ and since $x\leq 2x_J$ we get
$$
|x^{-\lambda }_Qu(x_Q,t_Q)-x_J^{-\lambda }u(x_J,t_J)|\leq Ax_J^{-\lambda }\leq 2^\lambda Ax^{-\lambda }.
$$

On the other hand, for some $z$ in the segment joining $x$ and $x_J$ and for some $s$ in the segment joining $t$ and $t_J$, we have that
$$
x^{-\lambda }u(x,t)-x_J^{-\lambda}u(x_J,t_J)=\partial _z(z^{-\lambda }u(z,t))(x_J-x)+\partial _s(x_J^{-\lambda }u(x_J,s))(t_J-t).
$$
Since $x\leq 2x_J$, from Proposition \ref{Prop2.2} it follows that
$$
|x^{-\lambda}u(x,t)-x_J^{-\lambda}u(x_J,t_J)|\leq Cx^{-\lambda }\|f\|_{BMO_{\rm o}(\mathbb{R})},
$$
and \eqref{3.1} is checked.

Next, we show that
\begin{equation}\label{3.2}
\sum_{J\subset Q, \, J\in G_{k+1}}|J|\leq \frac{C}{A}|Q| \, \|f\|_{BMO_{\rm o}(\mathbb{R})}.
\end{equation}
For that, we write
$$
\sum_{J\subset Q, J\in G_{k+1}}|J|\leq \frac{1}{A}\sum_{J\subset Q, J\in G_{k+1}}|J| \, x_J^\lambda |x_Q^{-\lambda}u(x_Q,t_Q)-x_J^{-\lambda }u(x_J,t_J)|,
$$
and use the following decomposition for every $J\in G_{k+1}$, $J\subset Q$,
\begin{align*}
x_Q^{-\lambda}u(x_Q,t_Q)-x_J^{-\lambda }u(x_J,t_J)&=x_Q^{-\lambda }P_{t_Q}^\lambda (f-f_Q)(x_Q)-x_J^{-\lambda}P_{t_J}^\lambda (f-f_J)(x_J)\\
&\quad +x_Q^{-\lambda}f_Q[P_{t_J}^\lambda (1)(x_Q)-1]- x_J^{-\lambda }f_J[P_{t_J}^\lambda (1)(x_J)-1]\\
&\quad +[x_Q^{-\lambda }f_Q-x_J^{-\lambda }f_J]\\
&=:\sum_{i=1}^5H_i (J) .
\end{align*}
Let $J\in G_{k+1}$, $J\subset Q$. According to \eqref{3.6.1} and since $x_J\leq 2x_Q$ we obtain
$$
x_J^\lambda |H_1(J)+H_2(J)|\leq C\Big[\Big(\frac{x_J}{x_Q}\Big)^\lambda +1\Big]\|f\|_{BMO_{\rm o}(\mathbb{R})}\leq C\|f\|_{BMO_{\rm o}(\mathbb{R})}.
$$
By using \eqref{dif1} and \eqref{dif2} and proceeding as in \eqref{J3} and in \eqref{J3b} we also get 
$$
x_J^\lambda |H_3(J)+H_4(J)|\leq C(|f_Q||P_{t_Q}^\lambda (1)(x_Q)-1|+|f_J||P_{t_J}^\lambda (1)(x_J)-1|)\leq C\|f\|_{BMO_{\rm o}(\mathbb{R})}.
$$

We now study $H_5$. We can write
\begin{align*}
x_J^\lambda H_5&\leq \Big|\Big[\Big(\frac{x_J}{x_Q}\Big)^{\lambda }-1\Big]f_Q\Big|+|f_Q-f_J|.
\end{align*}

If $x_Q\leq t_Q$, then, as in \eqref{J3},  
$$
\Big|\Big[\Big(\frac{x_J}{x_Q}\Big)^{\lambda }-1\Big]f_Q\Big|\leq C|f_Q|\leq C\|f\|_{BMO_{\rm o}(\mathbb{R})}.
$$

In the case that $x_Q>t_Q$, since $J\subset Q$, it follows that $x_J/x_Q\subset (1/2,3/2)$, and then, by applying the mean value theorem we get
$$
\Big|\Big[\Big(\frac{x_J}{x_Q}\Big)^{\lambda }-1\Big]f_Q\Big|\leq C\frac{|x_Q-x_J|}{x_Q}|f_Q|\leq C\frac{t_Q}{x_Q}|f_Q|\leq C\frac{x_Q+t_Q}{x_Q}\|f\|_{BMO_{\rm o}(\mathbb{R})}\leq C\|f\|_{BMO_{\rm o}(\mathbb{R})}.
$$

By combining the above estimates we conclude that 
\begin{align*}
\sum_{J\subset Q, J\in G_{k+1}}|J|&
\leq \frac{C}{A}\sum_{J\subset Q, J\in G_{k+1}}|J|(\|f\|_{BMO_{\rm o}(\mathbb{R})}+|f_Q-f_J|)\\
&\leq \frac{C}{A}\Big(|Q|\|f\|_{BMO_{\rm o}(\mathbb{R})}+\sum_{J\subset Q, J\in G_{k+1}}\int_J|f(y)-f_Q|dy\Big)\\
&\leq \frac{C}{A}\Big(|Q|\|f\|_{BMO_{\rm o}(\mathbb{R})}+\int_Q|f(y)-f_Q|dy\Big)\leq \frac{C}{A}|Q|\|f\|_{BMO_{\rm o}(\mathbb{R})},
\end{align*}
and \eqref{3.2} is established. 

By choosing $A:=2C(1+\|f\|_{BMO_{\rm o}(\mathbb{R})})$ we obtain
\begin{equation}\label{mitad}
\sum_{J\subset Q, \, J\in G_{k+1}}|J|\leq \frac{|Q|}{2},
\end{equation}
 for every $Q\in G_k$, $k\in \mathbb{N}$.
 
Another helpful property is the following. Let $0 \leq a < b < \infty$ and $0<c<d<\infty$. For every $\alpha \in \mathbb{R}$, we have that
\begin{align}\label{A.5}
2\int_c^d\int_a^bt\nabla_{\lambda,y}(P_t^\lambda(x,y))\cdot \nabla_{\lambda,y}(P_t^\lambda(f)(y)-\alpha y^\lambda)dydt&  \nonumber \\
&\hspace{-7cm} =\int_a^b \Big[t\partial_t(P_t^\lambda(x,y))(P_t^\lambda(f)(y)-\alpha y^\lambda )+tP_t^\lambda(x,y)\partial_t(P_t^\lambda(f)(y))  \nonumber\\
& \hspace{-7cm} \quad \quad -  P_t^\lambda(x,y)(P_t^\lambda(f)(y)-\alpha y^\lambda)\Big]_{t=c}^{t=d} dy\nonumber\\
& \hspace{-7cm}\quad + \int_c^dt\Big[P_t^\lambda(x,y)D_{\lambda ,y}(P_t^\lambda(f)(y))+D_{\lambda ,y}(P_t^\lambda(x,y)) (P_t^\lambda(f)(y)-\alpha y^\lambda)\Big]_{y=a}^{y=b} dt\nonumber\\
&\hspace{-7cm} :=\int_a^b \Big[H_\alpha (x,y,t)\Big]_{t=c}^{t=d}dy+\int_c^d\Big[V_\alpha (x,y,t)\Big]_{y=a}^{y=b}dt,\quad x\in (0,\infty ). 
\end{align}

Indeed, by integrating by parts we get, for all $x \in (0,\infty)$,
\begin{align*}
& \int_{c}^{d} \int_a^b
t D_{\lambda,y} (P_t^\lambda (x,y))D_{\lambda,y} (P_t^\lambda(f)(y))dy \,dt \\
& \qquad \qquad = \frac{1}{2} \Big\{ 
\int_{c}^{d} 
\Big[t  P_t^\lambda (x,y) \, 
D_{\lambda,y} (P_t^\lambda(f)(y))
+ t D_{\lambda,y} (P_t^\lambda (x,y)) \, 
P_t^\lambda(f)(y) \Big]_{y=a}^{y=b} \, dt \\
& \qquad \qquad \qquad + \int_{c}^{d} \int_a^b t  \Big(
P_t^\lambda (x,y) \, B_{\lambda,y}  (P_t^\lambda(f)(y)) 
+  B_{\lambda,y}(P_t^\lambda (x,y)) \,   P_t^\lambda(f)(y) \Big) \, dy \, dt
\Big\}.
\end{align*}
Also we have that
\begin{align*}
& \int_{c}^{d} \int_a^b
t \partial_t (P_t^\lambda (x,y)) \, 
\partial_t (P_t^\lambda(f)(y)) \, dy \, dt \\
& \qquad = \frac{1}{2} 
\int_{c}^{d} \int_a^b t \Big\{ 
\partial_t^2 [P_t^\lambda (x,y) P_t^\lambda(f)(y)]
- \partial_t^2 (P_t^\lambda (x,y)) P_t^\lambda(f)(y) \\
& \qquad \qquad  - P_t^\lambda (x,y) \partial_t^2  (P_t^\lambda(f)(y))
\Big\} \, dy \, dt, \quad x \in (0,\infty).
\end{align*}

By \cite[Lemma 2.2 and (2.12)]{BSt} 
$$(\partial_t^2-B_{\lambda,y})P_t^\lambda(f)(y)=0
\quad \text{and} \quad 
(\partial_t^2-B_{\lambda,y})P_t^\lambda(x,y)=0,
\quad t,x,y\in (0,\infty),$$ 
we obtain, for $x\in (0,\infty)$,
\begin{align*}
& \int_c^d\int_a^bt\nabla_{\lambda,y}(P_t^\lambda(x,y))\cdot \nabla_{\lambda,y}(P_t^\lambda(f)(y))dydt \\
& \qquad =\frac{1}{2}\Big( \int_a^b\int_c^d t\partial_t ^2[P_t^\lambda(x,y)P_t^\lambda(f)(y)]dtdy   \\
& \qquad \qquad+   \int_c^d\left[tP_t^\lambda(x,y)D_{\lambda,y} (P_t^\lambda(f)(y))+tD_{\lambda,y}(P_t^\lambda(x,y)) P_t^\lambda(f)(y)\right]_{y=a}^{y=b} dt\Big) \\
& \qquad =\frac{1}{2}\Big( \int_a^b\left[ t\partial_t[P_t^\lambda(x,y)P_t^\lambda(f)(y)]- P_t^\lambda(x,y)P_t^\lambda(f)(y)\right]_{t=c}^{t=d}dy  \\
& \qquad \qquad+   \int_c^d\left[tP_t^\lambda(x,y)D_{\lambda,y} (P_t^\lambda(f)(y))+tD_{\lambda,y}(P_t^\lambda(x,y)) P_t^\lambda(f)(y)\right]_{y=a}^{y=b} dt\Big) \\
& \qquad =\frac{1}{2}\Big( \int_a^b\left[ t\partial_t(P_t^\lambda(x,y))P_t^\lambda(f)(y)+tP_t^\lambda(x,y)\partial_t(P_t^\lambda(f)(y))- P_t^\lambda(x,y)P_t^\lambda(f)(y)\right]_{t=c}^{t=d}dy  \\
& \qquad \qquad+   \int_c^d\left[tP_t^\lambda(x,y)D_{\lambda,y} (P_t^\lambda(f)(y))+tD_{\lambda,y}(P_t^\lambda(x,y)) P_t^\lambda(f)(y)\right]_{y=a}^{y=b} dt\Big).
\end{align*}
Then, by taking into account that $\nabla_{\lambda ,y}(\alpha y^\lambda)=(0,0)$, for every $\alpha\in\mathbb R$, we obtain  \eqref{A.5}.

In order to prove Theorem \ref{Th1.3}, ({\it {ii}}), by using (\ref{A.2}), we can write
\begin{equation*}
f(x)
= 2 \lim_{n \to \infty} \int_{2^{-n}}^{2^{n}} \int_0^\infty t \nabla_{\lambda,y} (P_t^\lambda (x,y)) \cdot \nabla_ {\lambda,y}(P_t^\lambda (f)(y)) \, dy \, dt, \quad
\text{in } L^2(0,\infty).
\end{equation*}

For every $n \in \N$ we define the sets
\begin{align*}
U_n&=[0,2)\times (2^{-n}, 2),\\
W_n&=[0,\infty) \times (2^{-n},2^n)\setminus U_n,
\end{align*}
and, for every $k \in \N$ and $Q \in G_k$, $\Sigma_{Q,n}:= \Sigma_Q \cap U_n.$
Note that if $k,n\in \mathbb{N}$ and $k >n$, then $\Sigma_{Q,n} = \emptyset$, for each $Q \in G_k$.

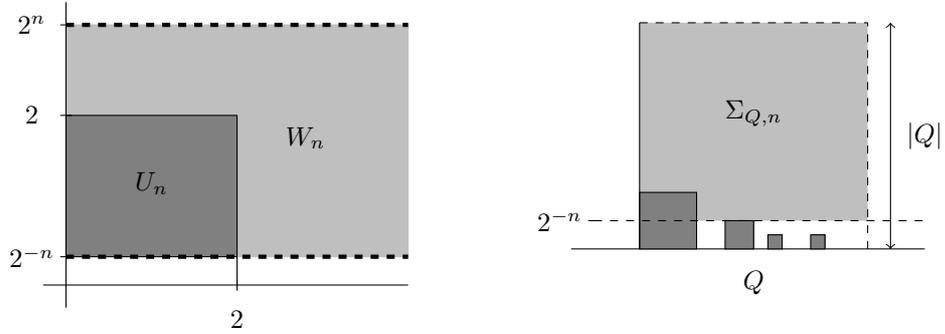
\begin{figure}[h!]
\hspace{-6cm} 
\begin{minipage}{5cm}
 \begin{tikzpicture}[scale=1.5]
                \draw[fill=gray](0,0.25) rectangle (1.5,1.5); % U_n y W_n
	     \node at (0.75,0.9){$U_n$};
	     \draw[fill=lightgray, lightgray](0,1.5) rectangle (1.5,2.3);
 	     \draw[fill=lightgray, lightgray](1.5,0.25) rectangle (3,2.3);
	     \node at (2.1,1.3){$W_n$};

	     \draw[-] (-0.2,0) -- (3,0); %ejes          
                \draw[-] (0,-0.2) -- (0,2.5) ;

                \draw[-] (0,1.5) -- (1.5,1.5); %cubo
                \draw[-] (-0.05,1.5) -- (0.05,1.5);
                \node at (-0.3,1.5) {$2$};
                \draw[-] (1.5,0) -- (1.5,1.5);
                \draw[-] (1.5,-0.05) -- (1.5,0.05);
                \node at (1.5,-0.3) {$2$};
                
                \draw[dashed, ultra thick] (0,2.3) -- (3,2.3); %línea $2^n$
                \draw[-] (-0.05,2.3) -- (0.05,2.3);
                \node at (-0.3,2.3) {$2^{n}$};
                                
                \draw[dashed, ultra thick] (0,0.25) -- (3,0.25); %línea 2^{-n}
                \draw[-] (-0.05,0.25) -- (0.05,0.25);
                \node at (-0.3,0.25) {$2^{-n}$};

 \end{tikzpicture}\end{minipage}

\vspace{-4.2cm} \hspace{8cm}
    \begin{minipage}{5cm}
%        \begin{center}
            \begin{tikzpicture}[scale=1.5]
                	\draw[-] (-0.6,0) -- (2.5,0);    %ejes    

 		 \draw[<->] (2.2,0) -- (2.2,2);      
		\node at (2.5,1) {$|Q|$};
		\node at (1,-0.3) {$Q$};

		\path[fill=lightgray](0,0.5)--(0.5,0.5)--(0.5,0.25)--(1.5,0.25)--(2,0.25)--(2,2)--(0,2)--(0,0.5);
		\draw[fill=gray](0,0) rectangle (0.5,0.5);
		\draw[fill=gray](0.75,0) rectangle (1,0.25);
		\draw[fill=gray](1.125,0) rectangle (1.25,0.125);
		\draw[fill=gray](1.5,0) rectangle (1.625,0.125);
		\node at (1,1.2){$\Sigma _{Q,n}$};	
		\draw[-] (0,0) -- (0,2) ; %tienda de Q 
        \draw[dashed] (0,2) -- (2,2);
        \draw[dashed] (2,0) -- (2,2);
		\draw[dashed] (-0.4,0.25) -- (2.5,0.25);%línea 2^{-n}
  	          \draw[dashed] (-0.45,0.25) -- (-0.35,0.25); 
	          \node at (-0.7,0.25) {$2^{-n}$}; 
 \end{tikzpicture}
%        \end{center}
    \end{minipage}
%\caption{Global regions}\label{gr}
\label{fig:regions}
\end{figure}
%\begin{figure}[h!]
%    \begin{minipage}{7.5cm}
%        \begin{center}
%            \begin{tikzpicture}[scale=2]
%                \draw[-] (-0.2,0) -- (3.1,0);           
%                \draw[-] (0,-0.2) -- (0,3.1) ;%
%
%                \draw[-] (0,2) -- (2,2);
%               \draw[-] (-0.05,2) -- (0.05,2);
%                \node at (-0.3,2) {$2$};
                
%                \draw[-] (2,0) -- (2,2);
%                \draw[-] (2,-0.05) -- (2,0.05);
%                \node at (2,-0.3) {$2=\dfrac{2 \cdot 2^n}{2^n}$};
                
%                \draw[dashed] (0,2.5) -- (3,2.5);
%                \draw[-] (-0.05,2.5) -- (0.05,2.5);
%                \node at (-0.3,2.5) {$2^{n}$};
                                
%                \draw[dashed] (0,0.25) -- (3,0.25);
%                \draw[-] (-0.05,0.25) -- (0.05,0.25);
%                \node at (-0.3,0.25) {$2^{-n}$};   
                
%                \draw[-] (0.25,-0.05) -- (0.25,0.05);
%                \draw[-] (0.25,0.22) -- (0.25,0.28);
%                \node at (0.25,-0.3) {$\dfrac{1}{2^n}$};
                
%                \draw[-] (0.5,-0.05) -- (0.5,0.05);
%                \draw[-] (0.5,0.22) -- (0.5,0.28);
%                \node at (0.5,-0.3) {$\dfrac{2}{2^n}$};
                
%                \draw[-] (0.75,-0.05) -- (0.75,0.05);
%                 \draw[-] (0.75,0.22) -- (0.75,0.28);
%                \node at (0.75,-0.3) {$\dfrac{3}{2^n}$};
%            \end{tikzpicture}
%        \end{center}
%    \end{minipage}
%\caption{Global regions}\label{gr}
%\label{fig:regions}
%\end{figure}

% \newpage
We obtain, for each $n\in \mathbb{N}$,
\begin{align*}
2\int_{2^{-n}}^{2^{n}} \int_0^\infty t \nabla_{\lambda,y} (P_t^\lambda (x,y)) \, \cdot \, \nabla_{\lambda,y} (P_t^\lambda (f)(y)) \, dy \, dt &\\
&\hspace{-4cm} =2\Big(\int_{U_n}+\int_{W_n}\Big)t \nabla_{\lambda,y} (P_t^\lambda (x,y)) \, \cdot \, \nabla_{\lambda,y} (P_t^\lambda (f)(y)) \, dy \, dt\\
& \hspace{-4cm}=2 \Big(\int_{\bigcup_{Q\in \cup_{k=0}^n G_k}\Sigma_{Q,n}}
+ \int_{W_n} \Big) t\nabla_{\lambda,y} (P_t^\lambda (x,y)) \, \cdot \, \nabla_{\lambda,y} (P_t^\lambda (f)(y)) \, dy \, dt \\
& \hspace{-4cm}=: G_{1,n}(x) + G_{2,n}(x), \quad x \in (0,\infty).
\end{align*}

Our objective is to establish that, there exists an increasing sequence $\{n_i\}_{i\in \mathbb{N}}$ of nonnegative integers such that
\begin{equation}\label{O1}
\lim_{i\rightarrow \infty}G_{1,n_i}(x)=g_1(x)+S_{\sigma _1, P^\lambda }(x),\quad \mbox{ a.e. }x\in (0,\infty),
\end{equation}
and
\begin{equation}\label{O2}
\lim_{i\rightarrow \infty}G_{2,n_i}(x)=g_2(x)+S_{\sigma _2, P^\lambda }(x),\quad \mbox{ a.e. }x\in (0,\infty),
\end{equation}
for certain $g_1$ and $g_2\in L^\infty(0,\infty)$ and $\sigma _1$ and $\sigma _2$ Carleson measures on $(0,\infty )\times (0,\infty )$ such that
$$
\|g_1\|_\infty +\|g_2\|_\infty +\|\sigma _1\|_\mathcal{C}+\|\sigma _2\|_\mathcal{C}\leq C(A+\|f\|_{BMO_{\rm o}(\mathbb{R})}),
$$
and thus we can conclude our result. 

Let $n\in \mathbb{N}$. First we deal with the function $G_{1,n}$.  We can write
$$
G_{1,n}(x)=\sum_{Q\in \bigcup_{k\in \mathbb{N}} G_k}2\int_{\Sigma _{Q,n}} t\nabla_{\lambda,y} (P_t^\lambda (x,y)) \, \cdot \, \nabla_{\lambda,y} [P_t^\lambda (f)(y)-c_Qy^\lambda] \, dy \, dt,
$$
where $c_Q:=y^{-\lambda} P_t^\lambda(f)(y)_{|_{(y,t)=(x_Q,t_Q)}}$.

Let $k\in \mathbb{N}$ and $Q\in G_k$. By taking into account (\ref{A.5}) it follows that the integral
$$
\int_{\Sigma_{Q,n}}t\nabla_{\lambda,y}(P_t^\lambda(x,y))\cdot\nabla_{\lambda,y}[P_t^\lambda(f)(y)-c_Qy^\lambda]dydt
$$
reduces to an integral over the boundary $\partial\Sigma_{Q,n}$ of $\Sigma_{Q,n}$.

We decompose this boundary in vertical and horizontal segments as follows. Let us denote by ${\mathcal V}_{Q,n}$ the set of vertical segments in $\partial\Sigma_{Q,n}\cap ([0,2]\times [2^{-n},2])$, by ${\mathcal H}_{Q,n}$ the set constituted by all horizontal segments in $\partial\Sigma_{Q,n}\cap ([0,2]\times (2^{-n},2])$ and those ones in $\partial\Sigma_{Q,n}\cap ([0,2]\times \{2^{-n}\})$ which belong to the boundary of some $Q'\subset Q$, $Q'\in G_{k+1}$ with $|Q'|=2^{-n}$ and finally we consider ${\mathcal H}_{Q,n}^0$ the set of all horizontal segments in $\partial\Sigma_{Q,n}\cap ([0,2]\times \{2^{-n}\})$ that are not in ${\mathcal H}_{Q,n}$.   

Indeed we can write
\begin{align*}
{\mathcal H}_{Q,n}&=\bigcup_{I\in \mathbb{I}_{Q,n}}(I\times \{|I|\}),\\
{\mathcal H}_{Q,n}^0&=\bigcup_{J\in \mathbb{I}_{Q,n}^0}(J\times \{2^{-n}\}),\\
{\mathcal V}_{Q,n}&=\bigcup_{K\in \mathbb{K}_{Q,n}}(\{a_K\}\times K),
\end{align*}
where, when $k\leq n$, $\mathbb{I}_{Q,n}$ is the set constituted by $Q$ and all intervals $I\subset Q$, $I\in G_{k+1}$ with $|I|\geq 2^{-n}$, $\mathbb{I}_{Q,n}^0$ contains the maximal dyadic intervals $J\subset Q\setminus \{I\subset Q, I\in G_{k+1},|I|\geq 2^{-n}\}$, $\mathbb{K}_{Q,n}$ is a finite set of dyadic intervals in $[2^{-n},2]$, and $a_K\in [0,2]$, for every $K\in \mathbb{K}_{Q,n}$. When $k>n$, we consider $\mathbb{I}_{Q,n}=\mathbb{I}_{Q,n}^0=\mathbb{K}_{Q,n}=\emptyset$.

According to \eqref{A.5} we have that
\begin{align}\label{G1}
 G_{1,n}(x)&=-\sum_{Q\in \bigcup_{k\in \mathbb{N}} G_k}\sum_{J\in \mathbb{I}_{Q,n}^0}\int_JH_{c_Q}(x,y,t)_{|t=2^{-n}}dy\nonumber\\
&\quad  +\sum_{Q\in \bigcup_{k\in \mathbb{N}}G_k}\left(\sum_{I\in \mathbb I_{Q,n}}\varepsilon_I\int_I H_{c_Q}(x,y,t)_{|t=|I|}dy +\sum_{K\in \mathbb K_{Q,n}}\varepsilon _K \int_KV_{c_Q}(x,y,t)_{|y=a_K}dt\right)\nonumber\\
& :=\mathfrak{g}_{1,n}(x)+\mathfrak{g}_{2,n}(x),\quad x\in (0,\infty ).
\end{align}
Here $\varepsilon_J=\pm 1$, $J\in \mathbb{I}_{Q,n}\cup\mathbb{K}_{Q,n}$. 

Next we show that
\begin{equation}\label{A.6}
\lim_{n\rightarrow \infty}\mathfrak{g}_{1,n}(x)=\sum_{Q\in\cup_{k\in\mathbb N}G_k}(f(x)-c_Qx^\lambda)\chi_{\partial\Sigma_Q\cap([0,2]\times\{0\})}(x) =:g_1(x),
\end{equation}
in $L^2(0,\infty )$.

We can write, for every $n\in \mathbb{N}$,
\begin{align*}
\mathfrak{g}_{1,n}(x)&=- \int_0^\infty\sum_{Q\in \cup_{\in \mathbb{N}}G_k}\chi_{\mathbb{I}_{Q,n}^0}(y)H_{c_Q}(x,y,t)_{|t=2^{-n}}dy\\
&=\int_0^\infty \Big(P_t^\lambda (x,y)\sum_{Q\in \cup_{k\in \mathbb{N}}G_k}\chi_{\mathbb{I}_{Q,n}^0}(y)[P_t^\lambda (f)(y)-c_Qy^\lambda ]\Big)_{|t=2^{-n}}dy\\
&\quad -\int_0^\infty \Big(t\partial _tP_t^\lambda (x,y)\sum_{Q\in \cup_{k\in \mathbb{N}}G_k}\chi_{\mathbb{I}_{Q,n}^0}(y)[P_t^\lambda (f)(y)-c_Qy^\lambda ]\Big)_{|t=2^{-n}}dy\\
&\quad -\int_0^\infty \Big(P_t^\lambda (x,y)\sum_{Q\in \cup_{k\in \mathbb{N}}G_k}\chi_{\mathbb{I}_{Q,n}^0}(y)t\partial _t(P_t^\lambda (f)(y))\Big)_{|t=2^{-n}}dy,\quad x\in (0,\infty ).
\end{align*}

According to (\ref{3.1}), for $k\in\mathbb N$ and $Q\in G_k,$
\begin{equation}\label{A.7*}
|P_t^\lambda(f)(y)-c_Qy^\lambda |=y^\lambda |y^{-\lambda}u(y,t)-x_Q^{-\lambda }u(x_Q,t_Q)|\leq C(A+\|f\|_{BMO_{\rm o}(\mathbb R)}),\quad(y,t)\in\Sigma_Q.
\end{equation}

By using (\ref{A.4.1}), since $f\in L^2(0,\infty)$,
$$\lim_{t\rightarrow 0^+}P_t^\lambda(f)(y)=f(y),\;\;\;\mbox{a.e.}\;\;y\in (0,\infty).$$

Then it follows that
\begin{equation}\label{A.7}
|f(y)-c_Qy^\lambda |\leq C(A+\|f\|_{BMO_{\rm o}(\mathbb R)}),\;\;\;\mbox{a.e.}\;\;y\in \partial\Sigma_Q\cap([0,2]\times\{0\}).
\end{equation}

We observe also that
$$
 \supp \Big(\sum_{Q\in \cup_{\in \mathbb{N}}G_k}\chi_{\mathbb{I}_{Q,n}^0}(y)[P_t^\lambda (f)(y)-c_Qy^\lambda ]_{|t=2^{-n}}-g_1(y)\Big)\subset [0,2],\quad n\in\mathbb N ,
$$
and 
\begin{equation}\label{converg}
\lim_{n\rightarrow \infty}\sum_{Q\in \cup_{k\in \mathbb{N}}G_k}\chi_{\mathbb{I}_{Q,n}^0}(y)[P_t^\lambda (f)(y)-c_Qy^\lambda ]_{|t=2^{-n}}=g_1(y),\quad \mbox{ a.e. }y\in (0,\infty ),
\end{equation}
(actually, we can assure that \eqref{converg} is true for all $y\in Q_0$ which is not a dyadic number). By the dominated convergence theorem, we get \eqref{converg} in $L^2(0,\infty)$.

According to \eqref{D7} and \eqref{A.3} we have, for each $n\in \mathbb{N}$,
\begin{align*}
\mathfrak{g}_{1,n}(x)-g_1(x)&=h_\lambda \Big(e^{-tz}h_\lambda \Big(\sum_{Q\in \cup_{\in \mathbb{N}}G_k}\chi_{\mathbb{I}_{Q,n}^0}(y)[P_t^\lambda (f)(y)-c_Qy^\lambda ]-g_1(y)\Big)(z)_{|t=2^{-n}}\Big)(x)\\
&\quad +h_\lambda ((e^{-tz}-1)_{|t=2^{-n}}h_\lambda (g_1)(z))(x)
\\
&\quad +h_\lambda \Big(tze^{-tz}h_\lambda\Big(\sum_{Q\in \cup_{k\in \mathbb{N}}G_k}\chi_{\mathbb{I}_{Q,n}^0}(y)[P_t^\lambda (f)(y)-c_Qy^\lambda ]-g_1(y)\Big)(z)_{|t=2^{-n}}\Big)(x)\\
&\quad -(t\partial _tP_t^\lambda (g_1)(x))_{|t=2^{-n}}\\
&\quad -h_\lambda\Big(e^{-tz}h_\lambda \Big(\sum_{Q\in \cup_{k\in \mathbb{N}}G_k}\chi_{\mathbb{I}_{Q,n}^0}(y)t\partial _t(P_t^\lambda (f)(y))\Big)(z)_{|t=2^{-n}}\Big)(x),\quad x\in (0,\infty ).
\end{align*}
Then, by taking into account the $L^2$-boundedness of $h_\lambda$ we get
\begin{align*}
\|\mathfrak{g}_{1,n}-g_1\|_2&\leq C\Big(\Big\|\sum_{Q\in \cup_{k\in \mathbb{N}}G_k}\chi_{\mathbb{I}_{Q,n}^0}(y)[P_t^\lambda (f)(y)-c_Qy^\lambda ]_{|t=2^{-n}}-g_1(y) \Big\|_2\\
&\quad +\|(e^{-tz}-1)_{|t=2^{-n}}h_\lambda (g_1)\|_2+\|(t\partial _tP_t^\lambda (g_1))_{|t=2^{-n}}\|_2\\
&\quad +\|(t\partial_t P_t^\lambda (f))_{t=2^{-n}}\|_2\Big),\quad n\in \mathbb{N}.
\end{align*}
From \eqref{D8}, \eqref{converg} and the dominated convergence theorem we obtain \eqref{A.6}. Note that by \eqref{A.7} we have that $\|g_1\|_{\infty}\leq C(A+\|f\|_{BMO_{\rm o}(\mathbb R)})$.

Now, let us show that there exists a Carleson measure $\sigma_1$ on $(0,\infty )^2$ such that 
\begin{equation}\label{sigma1}
\lim_{n\rightarrow \infty}\mathfrak{g}_{2,n}(x)=S_{\sigma_1,P^\lambda}(x),\quad x\in (0,\infty ).
\end{equation}

For that, we write, for each $n\in \mathbb{N}$,
\begin{align*}
\mathfrak{g}_{2,n}(x)&=\sum_{Q\in \bigcup_{k\in \mathbb{N}}G_k}\left(\sum_{I\in \mathcal{H}_{Q,n}}\int_I P_t^\lambda(x,y)M_{Q,1}(y,t)dy_I +\sum_{J\in \mathcal{V}_{Q,n}}\int_JP_t^\lambda (x,y)L_{Q,1}(y,t)dt_J\right)\\
&\quad +\sum_{Q\in \bigcup_{k\in \mathbb{N}}G_k}\left(\sum_{I\in \mathcal{H}_{Q,n}}\int_I M_{Q,2}(x,y,t)dy_I +\sum_{J\in \mathcal{V}_{Q,n}}\int_JL_{Q,2}(x,y,t)dt_J\right)\\
&:=F_{1,n}(x)+F_{2,n}(x),\quad x\in (0,\infty ),
\end{align*}
where, for every $Q\in \cup_{k\in \mathbb{N}}G_k$,
\begin{align}\label{M}
&M_{Q,1}(y,t)
:=t\partial_tP_t^\lambda(f)(y)-[P_t^\lambda(f)(y)-c_Qy^\lambda],\\
&L_{Q,1}(y,t)
:=tD_{\lambda,y}P_t^\lambda(f)(y),\nonumber\\
&M_{Q,2}(x,y,t)
:=t\partial_t(P_t^\lambda(x,y))[P_t^\lambda(f)(y)-c_Qy^\lambda],\nonumber
\end{align}
and
$$L_{Q,2}(x,y,t)
:=tD_{\lambda,y}P_t^\lambda(x,y)[P_t^\lambda(f)(y)-c_Qy^\lambda].$$
Let us consider
$$
{\mathcal H}_{Q}
:=\{\mbox{horizontal segments in } \partial\Sigma_{Q}\cap ([0,2]\times(0,2])\}, 
$$
and
$${\mathcal V}_{Q}
:=\{\mbox{vertical segments in } \partial\Sigma_{Q}\}.
$$

By \eqref{mitad} and according to \cite[p. 346]{Ga} the measures 
$$
\nu
:=\sum_{Q\in \cup_{k\in\mathbb{N}} G_k}\Big(\sum_{I\in {\mathcal H}_{Q}}dy_I+\sum_{J\in {\mathcal V}_{Q}}dt_J\Big),
$$ 
and
$$
\nu_n
:=\sum_{Q\in \cup_{k\in\mathbb{N}}G_k}\Big(\sum_{I\in {\mathcal H}_{Q,n}}dy_I+\sum_{J\in {\mathcal V}_{Q,n}}dt_J\Big),\quad n\in \mathbb{N},
$$ 
are Carleson measures. Moreover, we can write
$$
d\nu_n(y,t)=k_n(y,t)d\alpha (y,t) \mbox{ and }d\nu(y,t)=k(y,t)d\alpha (y,t),
$$
 for certain positive measure $\alpha$ and nonnegative functions $k_n$ and $k$ such that $k_n\uparrow k$, as $n\rightarrow \infty$, pointwisely. Then, $\|\nu_n\|_{\mathcal C}\leq \|\nu\|_{\mathcal C}$, $n\in \mathbb{N}$, and by the monotone convergence theorem,
$$
\lim_{n\rightarrow \infty}S_{\nu _n,P^\lambda }(x)=S_{\nu, P^\lambda}(x),\quad x\in (0,\infty ).
$$

We now define 
$$
 \mu_n(y,t):= \sum_{Q\in \cup_{k\in \mathbb{N}} G_k} \Big(\sum_{I\in{\mathcal H}_{Q,n}}M_{Q,1}(y,t) dy_I +\sum_{J\in{\mathcal V}_{Q,n}}L_{Q,1}(y,t)dt_J\Big), \quad n\in\mathbb N, 
$$
and
$$
 \mu(y,t):= \sum_{Q\in \cup_{k\in \mathbb{N}} G_k} \Big(\sum_{I\in{\mathcal H}_{Q}}M_{Q,1}(y,t) dy_I 
  +  \sum_{J\in{\mathcal V}_{Q}}L_{Q,1}(y,t)dt_J\Big).
$$

From Proposition \ref{Prop2.2}  and \eqref{A.7*} it follows that, for every $Q\in \bigcup_{k=0}^\infty G_k$,
$$
|M_{Q,1}(y,t)|\leq C(A+\|f\|_{BMO_{\rm o}(\mathbb R)}), \quad (y,t)\in \Sigma _Q,
$$
and 
$$
|L_{Q,1}(y,t)|\leq C\|f\|_{BMO_{\rm o}(\mathbb R)},\quad y,t\in (0,\infty).
$$
Then the measures $\mu$ and $\mu_n$, $n\in\mathbb N$, are Carleson measures, satisfying that
$$\|\mu_n\|_{\mathcal{C}}\leq C(A+\|f\|_{BMO_{\rm o}(\mathbb R)}),\;\;\;n\in\mathbb N,$$
$$\|\mu\|_{\mathcal{C}}\leq C(A+\|f\|_{BMO_{\rm o}(\mathbb R)}),$$
and, by Theorem \ref{Th1.3}, $(i)$ and the dominated convergence theorem
$$
\lim_{n\rightarrow\infty}S_{\mu_n, P^\lambda }(x)=S_{\mu, P^\lambda }(x),\quad x\in (0,\infty ).
$$

Note that, for every $n\in\mathbb N$, $F_{1,n}(x)=S_{\mu_n,P^\lambda}(x)$, $x\in (0,\infty )$. Then, 
$$
\lim_{n\rightarrow\infty}F_{1,n}(x)=S_{\mu,P^\lambda}(x),\quad x\in (0,\infty).
$$

Now we study  $F_{2,n}$, $n\in \mathbb{N}$.
By Lemma \ref{D4} and (\ref{D3}) we get
$$
|t\partial_tP_t^\lambda(x,y)|+|tD_{\lambda,y}P_t^\lambda(x,y)|\leq C\frac{t}{(x-y)^2+t^2},\quad x,y,t\in (0,\infty).
$$

Then, by taking into account (\ref{A.7*}) and proceeding as in \cite[p. 25 and 26]{Wi} (see also \cite[p. 2088]{CDLSY}), we get, for every $n\in\mathbb N$, a Carleson measure $\rho_n$ such that $F_{2,n}=S_{\rho_n,P^\lambda}$, and a Carleson measure $\rho$, such that
$$
\sum_{Q\in\bigcup_{k\in \mathbb{N}}G_k}\left(\sum_{I\in \mathcal{H}_Q}\int_I M_{Q,2}(x,y,t)dy_I +\sum_{J\in \mathcal{V}_Q}\int_JL_{Q,2}(x,y,t)dt_J\right)=S_{\rho ,P^\lambda}(x),\quad x\in (0,\infty ).
$$
We have that,
$$
\lim_{n\rightarrow\infty}S_{\rho_n,P^\lambda}(x)=S_{\rho,P^\lambda}(x),\quad x\in (0,\infty ),
$$
and $\|\rho \|_{\mathcal{C}}\leq C(A+\|f\|_{BMO_{\rm o}(\mathbb{R})})$. Then we obtain \eqref{sigma1} for $\sigma_1=\mu+\rho$. From \eqref{G1}, \eqref{A.6} and \eqref{sigma1} we can find an increasing sequence $\{n_i\}_{i\in \mathbb{N}}$ of nonnegative integers such that
$$
\lim_{i\rightarrow \infty}G_{1,n_i}(x)=g_1(x)+S_{\mu +\rho,P^\lambda }(x),\quad \mbox{ a.e. }x\in (0,\infty ).
$$
Note that $\|g_1\|_\infty +\|\sigma _1\|_{\mathcal{C}}\leq C(A+\|f\|_{BMO_{\rm o}(\mathbb{R})})$, so \eqref{O1} is thus established. 

We now deal with $G_{2,n}$, $n\in \mathbb{N}$. Let $M>2$. We define
$$
W_{n,M}:=\Big\{(x,t)\in W_n: x\in (0,M)\Big\}, \quad n\in\mathbb N.
$$

By (\ref{A.5}) we have that
\begin{align*}
 G_{2,n,M}(x) 
 & :=2\int_{W_{n,M}}t\nabla_{\lambda,y}(P_t^\lambda(x,y))\cdot\nabla_{\lambda,y}[P_t^\lambda(f)(y)-c_0y^\lambda]dydt \\
&=\int_0^MH_{c_0}(x,y,t)_{|t=2^n}dy-\int_2^MH_{c_0}(x,y,t)_{|t=2^{-n}} dy-\int_0^2 H_{c_0}(x,y,t)_{|t=2}dy\\
& \quad  -\int_2^{2^n}V_{c_0}(x,y,t)_{|y=0} dt  -\int_{2^{-n}}^2V_{c_0}(x,y,t)_{|y=2} dt+\int_{2^{-n}}^{2^n}V_{c_0}(x,y,t)_{|y=M} dt \\
& =:\sum_{i=1}^6I_{i,n,M}(x),\;\;\;x\in (0,\infty)\;\mbox{and}\;n\in\mathbb N,
\end{align*}
where $c_0:=x_{Q_0}^{-\lambda}u(x_{Q_0},t_{Q_0})$. Observe that, actually $I_{3,n,M}$ is independent of $n$ and $M$ and $I_{4,n,M}$ and $I_{5,n,M}$ do not depend on $M$.

First, we note that $I_{4,n,M}(x)=0$, $n\in \mathbb{N}$. Indeed, by Lemma \ref{D4} and \eqref{D2} it follows that
$$
\lim_{y\rightarrow 0^+}P_t^\lambda(x,y)=\lim_{y\rightarrow 0^+}D_{\lambda ,y}P_t^\lambda (x,y)=0,\quad x,t\in (0,\infty ).
$$
Then, by taking into account Proposition \ref{Prop2.2} and \eqref{A.7*} for $Q=Q_0$ it follows that $V_{c_0}(x,y,t)_{|y=0}=0$, $x,t\in (0,\infty )$.

On the other hand, we have that
\begin{align*}
I_{3,n,M}(x)
& =  -\int_0^2( P_t^\lambda(x,y)M_{Q_0,1}(y,t))_{|t=2} dy -\int_0^2M_{Q_0,2}(y,t)_{|t=2} dy\\
&:=I_3^1(x)+I_3^2(x),\quad x\in (0,\infty ).
\end{align*}
Here $M_{Q_0,1}$ and $M_{Q_0,2}$ are as in \eqref{M} with $Q=Q_0$. By considering Proposition \ref{Prop2.2} and \eqref{A.7*} we get a Carleson measure $\alpha_3^1$ such that $I_3^1(x)=S_{\alpha _3^1,P^\lambda }(x)$, $x\in (0,\infty )$ and  $\|\alpha_3^1\|_{\mathcal C}\leq C(A+\|f\|_{BMO_{\rm o}(\mathbb R)})$. Also, from Lemma \ref{D4}, \eqref{D3} and \eqref{A.7*}
we deduce as above (see \cite[p. 25 and 26]{Wi}) that $I_3^{2}(x)=S_{\alpha_3^2,P^\lambda}(x)$, $x\in (0,\infty )$, for some Carleson measure $\alpha_3^2$ such that $\|\alpha_3^2\|_{\mathcal C}\leq C(A+\|f\|_{BMO_{\rm o}(\mathbb R)})$. Then 
 there exists a Carleson measure $\alpha_3=\alpha_3^1+\alpha_3^2$ on $(0,\infty )^2$ such that $I_{3,n,M}(x)=S_{\alpha_3,P^\lambda }(x)$, $x\in (0,\infty )$. 
 
In a similar way we can see that, for every $n\in \mathbb{N}$, there exists a Carleson measure $\alpha_{5,n}$ such that $I_{5,n,M}=S_{\alpha_{5,n},P^\lambda}$ and $\|\alpha_{5,n}\|_{\mathcal C}\leq C(A+\|f\|_{BMO_{\rm o}(\mathbb R)})$. And, as above, by Theorem \ref{Th1.3} and the dominated convergence theorem, there exists $\alpha_5\in \mathcal C$ such that $S_{\alpha_{5,n},P^\lambda}(x)\rightarrow S_{\alpha_5,P^\lambda}(x)$, as $n\rightarrow\infty$, for a.e. $x\in (0,\infty)$.

We have also that, for every $n\in \mathbb{N}$,
$$
\lim_{M\rightarrow \infty}I_{6,n,M}(x)=0,\quad x\in (0,\infty ).
$$

It is sufficient to note that by Lemma \ref{D4} and \eqref{D2} it follows that
\begin{align*}
 |I_{6,n,M}(x)| 
 & \leq  C M^\lambda\int_{2^{-n}}^{2^n}\frac{t(xM)^\lambda }{((x-M)^2+t^2)^{\lambda  +1}}\Big(1+\frac{1}{t^{2\lambda +2}}\int_0^1|f(z)|dz\Big)dt\\
 &\leq C\frac{x^\lambda M^{2\lambda }}{(|x-M|+2^{-n})^{2\lambda +1}}\int_{2^{-n}}^{2^n}\Big(1+\frac{1}{t^{2\lambda +2}}\Big)dt\leq \frac{C_{n,x}M^{2\lambda }}{(|x-M|+2^{-n})^{2\lambda +1}},\quad x\in (0,\infty),
\end{align*}
for every $n\in \mathbb{N}$ and for certain $C_{n,x}>0$.

On the other hand, for each $n\in \mathbb{N}$,
\begin{equation}\label{I_1}
\lim_{M\rightarrow \infty}I_{1,n,M}(x)=\int_0^\infty H_{c_0}(x,y,t)_{|t=2^n}dy,\quad x\in (0,\infty ),
\end{equation}
and 
\begin{equation}\label{I_2}
\lim_{M\rightarrow \infty}I_{2,n,M}(x)=\int_2^\infty H_{c_0}(x,y,t)_{|t=2^{-n}}dy,\quad x\in (0,\infty ).
\end{equation}

It is sufficient to show that for each $n\in \mathbb{N}$, the integrals in the right side of \eqref{I_1} and \eqref{I_2} are absolutely convergent for every $x\in (0,\infty)$. Indeed, from Lemma \ref{D4} and (\ref{D2}) and Proposition \ref{Prop2.2}, we get
$$
|H_{c_0}(x,y,t)|\leq C\frac{t(xy)^\lambda }{((x-y)^2+t^2)^{\lambda +1}}(|P_t^\lambda (f)(y)|+y^\lambda +1),\quad x,y,t\in (0,\infty ).
$$

Since $f\in L^2(0,\infty)$, also $P_t^\lambda(f)\in L^2(0,\infty)$ and, by H\"older's inequality we deduce that
\begin{align*}
\int_0^\infty |H_{c_0}(x,y,t)|dy&\leq Ctx^\lambda \int_0^\infty\frac{y^\lambda}{((x-y)^2+t^2)^{\lambda +1}}(|P_t^\lambda(f)(y)|+y^\lambda +1)dy   \\
& \leq Ctx^\lambda\Big\{\Big(\int_0^\infty\frac{y^{2\lambda}}{((x-y)^2+t^2)^{2\lambda +2}}dy\Big)^{1/2}\|P_t^\lambda(f)\|_2\\
& \qquad  \qquad + \int_0^\infty \frac{y^\lambda+y^{2\lambda}}{((x-y)^2+t)^{\lambda +1}}dy\Big\}<\infty,\;\;x,t\in (0,\infty).
\end{align*}

We conclude that, for every $n\in \mathbb{N}$, 
\begin{align*}
G_{2,n}(x) &= \int_0^\infty H_{c_0}(x,y,t)_{|t=2^n}dy-\int_2^\infty H_{c_0}(x,y,t)_{|t=2^{-n}}dy\\
&\quad +S_{\alpha _3,P^\lambda}(x)+S_{\alpha_{5,n},P^\lambda}(x),\quad x\in (0,\infty).
\end{align*}

In order to finish the proof we are going to show that
\begin{equation}\label{H1}
\lim_{n\rightarrow \infty}\int_0^\infty H_{c_0}(x,y,t)_{|t=2^n}dy=c_0x^\lambda ,\quad x\in (0,\infty ),
\end{equation}
and
\begin{equation}\label{H2}
\lim_{i\rightarrow \infty}\int_2^\infty H_{c_0}(x,y,t)_{|t=2^{-n_i}}=c_0x^\lambda \chi _{(2,\infty)}(x),\quad \mbox{ a.e. }x\in (0,\infty ),
\end{equation}
for certain increasing sequence $\{n_i\}_{i\in \mathbb{N}}$ of nonnegative integers. Thus we obtain that
$$
\lim_{i\rightarrow \infty }G_{2,n_i}(x)=c_0x^\lambda \chi _{(0,2)}(x)+S_{\alpha_3+\alpha_5,P^\lambda}(x),\quad \mbox{ a.e. }x\in (0,\infty ),
$$
and we get \eqref{O2} with $g_2(x)=c_0x^\lambda \chi _{(0,2)}(x)$, $x\in (0,\infty)$, and $\sigma_2=\alpha_3+\alpha_5$.

Since $P_t^\lambda (y^\lambda )(x)=x^\lambda$, $x,t\in (0,\infty )$ (\cite[p. 455]{BCaFR}),   
$$
\int_0^\infty \partial_t P_t^\lambda(x,y)y^\lambda dy=\partial_t\int_0^\infty P_t^\lambda(x,y)y^\lambda dy=0, \quad x,t\in (0,\infty).
$$
The derivation under the integral sign is justified because by Lemma \ref{D4} and \eqref{D2} it follows that
$$
\int_0^\infty |\partial_t P_t^\lambda(x,y)y^\lambda| dy\leq C\int_0^\infty\frac{x^\lambda y^{2\lambda}} {((x-y)^2+t^2)^{\lambda  +1}}dy <\infty,\quad x,t\in (0,\infty).
$$

Then, we can write
\begin{align*}
\int_0^\infty H_{c_0}(x,y,t)dy&=\int_0^\infty (t\partial_ t(P_t^\lambda(x,y))P_t^\lambda(f)(y)+tP_t^\lambda(x,y)\partial_tP_t^\lambda(f)(y)-P_t^\lambda(x,y)P_t^\lambda(f)(y))dy\\
&\quad +c_0x^\lambda ,\quad x,t\in (0,\infty ). 
\end{align*}

From Lemma \ref{D4} and \eqref{D3} we have that
\begin{align*}
\Big|\int_0^\infty (t\partial_ t(P_t^\lambda(x,y))P_t^\lambda(f)(y)+tP_t^\lambda(x,y)\partial_tP_t^\lambda(f)(y)-P_t^\lambda(x,y)P_t^\lambda(f)(y))dy\Big| & \\
&\hspace{-9cm} \leq C \int_0^\infty \frac{t}{(x-y)^2+t^2}\int_0^1\frac{t}{(y-z)^2+t^2}|f(z)|dzdy \\
&\hspace{-9cm} \leq \frac{C}{t}\int_0^\infty \frac{t}{(x-y)^2+t^2}\int_0^1|f(z)|dzdy\\
&\hspace{-9cm} \leq C\frac{\|f\|_{BMO_{\rm o}(\mathbb{R})}}{t}\Big(\int_0^{2x}\frac{dy}{t}+\int_{2x}^\infty \frac{t}{(y+t)^2}dy\Big)\\
&\hspace{-9cm} \leq \frac{C}{t}\Big(\frac{x}{t}+\frac{t}{x+t}\Big),\quad x, t\in (0, \infty).
\end{align*}
Note that, for each $x\in (0,\infty )$, the last term tends to zero as $t\rightarrow \infty$, and consequently we get \eqref{H1}.

On the other hand, we have that
\begin{align*}
\int_2^\infty H_{c_0}(x,y,t)dy&=\int_2^\infty (t\partial_ t(P_t^\lambda(x,y))P_t^\lambda(f)(y)+tP_t^\lambda(x,y)\partial_tP_t^\lambda(f)(y)-P_t^\lambda(x,y)P_t^\lambda(f)(y))dy\\
&\quad -c_0t\partial _t\int_2^\infty P_t^\lambda (x,y)y^\lambda dy+c_0\int_2^\infty P_t^\lambda (x,y)y^\lambda dy\\
&=\int_2^\infty (t\partial_ t(P_t^\lambda(x,y))P_t^\lambda(f)(y)+tP_t^\lambda(x,y)\partial_tP_t^\lambda(f)(y)-P_t^\lambda(x,y)P_t^\lambda(f)(y))dy\\
&\quad +c_0t\partial_tP_t^\lambda (y^\lambda \chi _{(0,2)}(y))(x)+c_0x^\lambda -c_0P_t^\lambda (y^\lambda \chi _{(0,2)}(y))(x),\quad x,t\in (0,\infty ). 
\end{align*}

As above, by using Lemma \ref{D4} and \eqref{D2} it follows that
\begin{align*}
\Big|\int_2^\infty (t\partial_ t(P_t^\lambda(x,y))P_t^\lambda(f)(y)+tP_t^\lambda(x,y)\partial_tP_t^\lambda(f)(y)-P_t^\lambda(x,y)P_t^\lambda(f)(y))dy\Big| & \\
&\hspace{-9cm} \leq C \int_2^\infty P_t^\lambda (x,y)\int_0^1\frac{t(yz)^\lambda}{((y-z)^2+t^2)^{\lambda +1}}|f(z)|dzdy \\
&\hspace{-9cm} \leq Ct\int_2^\infty P_t^\lambda (x,y)y^\lambda \int_0^1\frac{|f(z)|}{(1+t^2)^{\lambda +1}}dzdy\\
&\hspace{-9cm} \leq C\|f\|_{BMO_{\rm o}(\mathbb{R})}t\int_0^\infty P_t^\lambda (x,y)y^\lambda dy\\
&\hspace{-9cm} \leq Cx^\lambda t,\quad x, t\in (0, \infty).
\end{align*}

Observe that, for each $x\in (0,\infty )$, the last term tends to zero as $t\rightarrow 0$. By taking into account \eqref{D8} and \eqref{A.4.1} we conclude that, there exists an increasing sequence $\{n_i\}_{i\in \mathbb{N}}$ of nonnegative integers such that
$$
\lim_{i\rightarrow \infty}\int_2^\infty H_{c_0}(x,y,t)_{|t=2^{-n_i}}dy=c_0x^\lambda(1-\chi_{(0,2)}(x))=c_0x^\lambda \chi _{(2,\infty)}(x),\quad\mbox{ a.e. } x\in (0,\infty ),
$$
and \eqref{H2} is proved.
\end{proof}

%%%%%%%%%%%%%%%%%%%%%%%%%%%%%%%%%%%%%%%%%%%%%%%%%%%%%%%%%%%%%%%%%%%

\end{document}